\documentclass[11pt]{article}
\usepackage{amsmath,amsthm,amsfonts,amssymb,amscd, amsxtra,mathtools,mathrsfs,commath,url,easyReview}
\usepackage{cite,url}
\usepackage{graphicx}
\usepackage[title]{appendix}

\usepackage{multirow}
\usepackage{multicol}   
\usepackage[margin=2.7 cm,nohead]{geometry}
\usepackage{bbm}
\usepackage{xcolor}
\usepackage{float}
\usepackage{threeparttablex,color,soul,colortbl,hhline,rotating,booktabs,longtable,multirow,makecell}
\definecolor{lightgray}{gray}{0.95}
\newcommand{\myscaletable}{0.6}
\newcommand{\lightgray}{0.95}
\usepackage{amssymb,amsfonts,amstext,amsmath}
\usepackage{graphicx,geometry}
\usepackage{listings}
\usepackage{algorithmic}

\usepackage{hyperref}

\newtheorem{theorem}{Theorem}[section]
\newtheorem{lemma}[theorem]{Lemma}

\newtheorem{proposition}[theorem]{Proposition}

\newtheorem{assumption}{Assumption}[section]
\newtheorem{algorithm}{Algorithm}

\numberwithin{assumption}{section}

\newcommand{\argmin}{{\rm arg}\!\min}

\newcommand{\N}{\mathbb{N}}
\newcommand{\R}{\mathbb{R}}

\newcommand{\K}{\mathbb{K}}

\newcommand{\Ls}{{\cal{L}}}

\newcommand{\D}{\mathcal{D}}
\newcommand{\dsd}{d_{SD}}
\newcommand{\xs}{x^{\ast}}
\newcommand{\lsd}{\lambda^{SD}}
\newcommand{\HF}{\nabla^2 F}
\newcommand{\ds}{\displaystyle}
\newcommand{\subsetinf}{\displaystyle\mathop{\subset}_{\infty}}

\newcommand{\myscale}{0.45}

\newcommand{\ty}{\tilde{y}_j^k}
\newcommand{\ts}{\tilde{s}_j^k}
\newcommand{\tq}{\tilde{q}_j^k}
\newcommand{\ta}{\tilde{a}_j^k}
\newcommand{\tB}{\tilde{B}_j^k}
\newcommand{\tb}{\tilde{b}_j^k}

\newcommand{\tg}{\tilde{\gamma}_j^k}

\newcommand{\tr}{{\rm trace}}

\newenvironment{proofof}[1]
{\vspace{8pt} \noindent{\bf Proof of #1.}}
{\qed\vspace{8pt}}

\newcommand{\bi}{\begin{itemize}}
\newcommand{\ei}{\end{itemize}}
\newcommand{\ba}{\begin{array}}
\newcommand{\ea}{\end{array}}

\usepackage{amssymb}

\begin{document}

\title{Global convergence of a BFGS-type algorithm for nonconvex multiobjective optimization problems}
\author{
L. F. Prudente\thanks{Instituto de Matem\'atica e Estat\'istica, Universidade Federal de Goi\'as,  CEP 74001-970, Goi\^ania, GO, Brazil, E-mails: {\tt  lfprudente@ufg.br},  {\tt  danilo\underline{ }souza@discente.ufg.br}. This work was funded by CAPES and CNPq (Grants 309628/2020-2, 405349/2021-1).}
\and
D. R. Souza\footnotemark[1]
}
\maketitle

\maketitle

\maketitle
\vspace{.5cm}

\noindent

\noindent
{\bf Abstract:} 
We propose a modified BFGS algorithm for multiobjective optimization problems with global convergence, even in the absence of convexity assumptions on the objective functions. Furthermore, we establish a local superlinear rate of convergence of the method under usual conditions. Our approach employs Wolfe step sizes and ensures that the Hessian approximations are updated and corrected at each iteration to address the lack of convexity assumption. Numerical results shows that the introduced modifications preserve the practical efficiency of the BFGS method.

\vspace{8pt}

\vspace{8pt}

\section{Introduction} 

Multiobjective optimization problems involve the simultaneous minimization of multiple objectives that may be conflicting.
The goal is to find a set of solutions that offer different trade-offs between these objectives, helping decision makers in identifying the most satisfactory solution. {\it Pareto optimality} is a fundamental concept used to characterize such solutions. A solution is said to be {\it Pareto optimal} if none of the objectives can be improved without deterioration to at least one of the other objectives. 

Over the last two decades, significant research has focused on extending iterative methods originally developed for single-criterion  optimization to the domain of multiobjective optimization, providing an alternative to scalarization methods \cite{Eichfelder2008,miettinen1999nonlinear}. This line of research was initiated by Fliege and Svaiter in 2000 with the extension of the steepest descent method~\cite{benar&fliege} (see also \cite{mauricio&benar}). Since then, several methods have been studied, including Newton~\cite{chuong2013newton,mauriciobenarfliege,mauricionewtonvetorial,wang2019extended,gonccalves2021globally},  quasi-Newton~\cite{ansary,math8040616,superlinearly,Morovati2017,POVALEJ2014765,QU2011397,QU2014503,Leandro&Danilo,Lapucci2023},  conjugate gradient~\cite{Goncalves2019,cgvo,GONCALVES2022127099}, conditional gradient \cite{Assuncao2021,Chen2023}, projected gradient~\cite{luis-jef-yun,fazzio2019convergence,fukuda2011convergence,fukuda2013inexact,mauricio&iusem}, and proximal  methods~\cite{bonnel2005proximal,ceng2010hybrid,ceng2007approximate,chuong2011generalized,chuong2011hybrid}.

Proposed independently by Broyden~\cite{Broyden70}, Fletcher~\cite{Fletcher70}, Goldfarb~\cite{Goldfarb70}, and Shanno~\cite{Shanno70} in 1970, the BFGS is the most widely used quasi-Newton method for solving unconstrained scalar-valued optimization problems. As a quasi-Newton method, it computes the search direction using a quadratic model of the objective function, where the Hessian is approximated based on first-order information. Powell~\cite{powell1976some} was the first to prove the global convergence of the BFGS method for convex functions, employing a line search that satisfies the Wolfe conditions. Some time later, Byrd and Nocedal~\cite{Tool} introduced additional tools that simplified the global convergence analysis, enabling the inclusion of backtracking strategies. For over three decades, the convergence of the BFGS method for nonconvex optimization remained an open question until Dai~\cite{dai2002convergence}, in the early 2000s, provided a counterexample showing that the method can fail in such cases (see also \cite{dai2013perfect,mascarenhas2004bfgs}).
Another research direction focuses on proposing suitable modifications to the BFGS algorithm that enable achieving global convergence for nonconvex general functions while preserving its desirable properties, such as efficiency and simplicity. Notable works in this area include those by Li and Fukushima~\cite{Li-Fu-MBFGS,Li-Fu-SIAM}.

The BFGS method for multiobjective optimization was studied in \cite{Leandro&Danilo,Lapucci2023,POVALEJ2014765,math8040616,superlinearly,Morovati2017,QU2011397,QU2014503}. However, it is important to note that, except for \cite{Leandro&Danilo,QU2011397}, the algorithms proposed in these papers are specifically designed for convex problems. The assumption of convexity is crucial to ensure that the Hessian approximations remain positive definite over the iterations, guaranteeing the well-definedness of these methods.
In \cite{QU2011397}, Qu {\it et al.} proposed a {\it cautious} BFGS update scheme based on the work \cite{Li-Fu-SIAM}.
This approach updates the Hessian approximations only when a given safeguard criterion is satisfied, resulting in a globally convergent algorithm for nonconvex problems.
In \cite{Leandro&Danilo}, Prudente and Souza proposed a BFGS method with Wolfe line searches which exactly mimics the classical BFGS method for single-criterion optimization. This variant is well defined even for general nonconvex problems, although global convergence cannot be guaranteed in this general case. 
Despite this, it has been shown to be globally convergent for strongly convex problems.

In the present paper, inspired by the work \cite{Li-Fu-MBFGS}, we go a step further than \cite{Leandro&Danilo} and introduce a modified BFGS algorithm for multiobjective optimization which possesses a global convergence property even without convexity assumption on the objective functions. Furthermore, we establish the local superlinear convergence of the method under certain conditions. Our approach employs Wolfe step sizes and ensures that the Hessian approximations are updated and corrected at each iteration to overcome the lack of convexity assumption. Numerical results comparing the proposed algorithm with the methods introduced in \cite{Leandro&Danilo,QU2011397} are discussed. Overall, the modifications made to the BFGS method to ensure global convergence for nonconvex problems do not compromise its practical performance.

The paper is organized as follows: Section~\ref{sec:pre} presents the concepts and preliminary results, Section~\ref{sec:algorithm} introduces the proposed modified BFGS algorithm and discusses its global convergence, Section~\ref{sec:superlinear} focuses on the local convergence analysis with superlinear convergence rate, Section~\ref{sec:numerical} presents the numerical experiments, and Section~\ref{sec:final} concludes the paper with some remarks. Throughout the main text, we have chosen to omit proofs that can be easily derived from existing literature to enhance overall readability. However, these proofs are provided in the Appendix for self-contained completeness.

\vspace{5pt}
\noindent{\bf Notation.} $\mathbb{R}$  and $\mathbb{R}_{++}$ denote the set of real numbers and the set of  positive real numbers, respectively. As usual, $\mathbb{R}^n$ and $\mathbb{R}^{n\times p}$ denote the set of $n$-dimensional real column vectors and the set of ${n\times p}$ real matrices,  respectively.
The identity matrix of size $n$ is denoted by $I_n$.  
$\|\cdot\|$ is the Euclidean norm.
If $u,v\in\R^n$, then $u\preceq v$ (or $\prec$) is to be understood in a componentwise sense, i.e., $u_i\leq v_i$ (or $<$) for all $i=1,\ldots,n$.
 For $B\in\R^{n\times n}$, $B \succ 0$ means that $B$ is positive definite.
In this case, $\langle \cdot,\cdot\rangle_B$ and $\|\cdot\|_B$ denote the $B$-energy inner product and the $B$-energy norm, respectively, i.e., for $u,v\in\R^n$, $\langle u,v\rangle_B:=u^\top B v$ and $\|u\|_B:=\sqrt{\langle u,u\rangle_B}$.
If $K = \{k_1, k_2, \ldots \} \subseteq \N$, with $k_j < k_{j+1}$ for all $j\in \N$, then  we denote $K  \subsetinf \N$.

\section{Preliminaries} \label{sec:pre}

In this paper, we focus on the problem of finding a {\it Pareto optimal} point of a continuously differentiable function $F: \mathbb{R}^n \rightarrow \mathbb{R}^m$. This problem can be denoted as follows:
\begin{equation} \label{vectorproblem}
\min_{x \in \mathbb{R}^n} F(x).
\end{equation}
A point $x^{\ast}\in\mathbb{R}^n$ is {\it Pareto optimal} (or {\it weak Pareto optimal}) of $F$ if there is no other point $x \in \mathbb{R}^n$ such that $F(x) \preceq F(x^{\ast})$ and $F(x) \neq F(x^{\ast})$ (or $F(x) \prec F(x^{\ast})$).
These concepts can also be defined locally. We say that $x^{\ast}\in\mathbb{R}^n$ is a {\it local Pareto optimal} (or  {\it local weak Pareto optimal}) point if there exists a neighborhood $U \subset \mathbb{R}^n$ of $x^{\ast}$ such that $x^{\ast}$ is Pareto optimal (or weak Pareto optimal) for $F$ restricted to $U$. A necessary condition (but not always sufficient) for the local weak Pareto optimality of $x^{\ast}$ is given by:
\begin{equation}\label{critical}
-(\mathbb{R}^m_{++}) \cap \text{Image}(JF(x^{\ast})) = \emptyset,
\end{equation}
where $JF(x^{\ast})$ denotes the Jacobian of $F$ at $x^{\ast}$. A point $x^{\ast}$ that satisfies \eqref{critical} is referred to as a {\it Pareto critical} point.
It should be noted that if $x\in\mathbb{R}^n$ is not Pareto critical, then there exists a direction $d\in\mathbb{R}^n$ such that $\nabla F_j(x)^\top  d < 0$ for all $j = 1, \ldots, m$. This implies that $d$ is a {\it descent direction} for $F$ at $x$, meaning that there exists $\varepsilon > 0$ such that $F(x + \alpha d) \prec F(x)$ for all $\alpha \in (0,\varepsilon]$. 
Let $\D: \mathbb{R}^n \times \mathbb{R}^n \rightarrow \mathbb{R}$ be defined as follows:
\begin{equation*} \label{dderiv}
\D(x,d) \coloneqq \max_{j = 1,\dots,m }\nabla F_j(x)^\top d.
\end{equation*}
The function $\D$ characterizes the descent directions for $F$ at a given point $x$. Specifically, if $\D(x,d) < 0$, then $d$ is a descent direction for $F$ at $x$. Conversely, if $\D(x,d) \geq 0$ for all $d \in \mathbb{R}^n$, then $x$ is a Pareto critical point. 

We define $F: \mathbb{R}^n \rightarrow \mathbb{R}^m$ as convex (or strictly convex) if each component $F_j: \mathbb{R}^n \rightarrow \mathbb{R}$ is convex (or strictly convex) for all $j = 1, \ldots, m$, i.e., for all $x,y \in\R^n$ and $t \in [0,1]$ (or $t \in (0,1)$), 
\begin{equation}\label{eq:convex}
 F((1-t)x+ty)\preceq (1-t)F(x)+tF(y) \quad (\mbox{or} \prec).
 \end{equation}
The following result establishes a relationship between the concepts of criticality, optimality, and convexity.
\begin{lemma}\cite[Theorem 3.1]{mauriciobenarfliege}\label{Theo3.1-Bennar-Fliege} The following statements  hold:
	\begin{itemize}
		\item[(i)]  if $x^{\ast}$ is local weak Pareto optimal, then $x^{\ast}$ is a Pareto critical point for $F$;
		\item[(ii)]  if $F$ is convex and $x^{\ast}$ is Pareto critical for $F$, then $x^{\ast}$ is weak Pareto optimal;
		\item[(iii)]  if $F$ is strictly convex  and $x^{\ast}$ is Pareto critical for $F$, then $x^{\ast}$ is Pareto optimal.
	\end{itemize}
\end{lemma}

The class of quasi-Newton methods used to solve \eqref{vectorproblem} consists of algorithms that compute the search direction $d(x)$ at a given point $x \in \mathbb{R}^n$ by solving the optimization problem:
\begin{equation}\label{probnewton}
\min_{d\in\R^n} \max_{j=1,\ldots,m} \nabla F_j (x)^\top d + \frac{1}{2}d^\top B_j d,
\end{equation}
where $B_j\in \R^{n\times n}$ serves as an approximation of $\HF_j(x)$ for all $j=1,\ldots,m$. If $B_j\succ 0$ for all $j=1,\ldots,m$, then the objective function of \eqref{probnewton} is strongly convex, ensuring a unique solution for this problem. We denote the optimal value of \eqref{probnewton} by $\theta(x)$, i.e.,
 \begin{equation}{\label{direction}}
d(x)\coloneqq\argmin_{d \in \R^n}\; \max_{j = 1,\dots,m }\nabla F_j (x)^\top d + \frac{1}{2}d^\top B_j d,
\end{equation} 
and
\begin{equation}{\label{theta}}
\theta(x) : = \max_{j = 1,\dots,m } \nabla F_j (x)^\top d(x) + \frac{1}{2}d(x)^\top B_j d(x).
\end{equation}
One natural approach is to use a BFGS--type formula, which updates the approximation $B_j$ in a way that preserves positive definiteness.
In the case where $B_j=I_n$ for all $j=1,\ldots,m$, $d(x)$ represents the steepest descent direction (see \cite{benar&fliege}). Similarly, if $B_j=\HF_j(x)$ for all $j=1,\ldots,m$, $d(x)$ corresponds to the Newton direction (see \cite{mauriciobenarfliege}).

In the following discussion, we assume that $B_j\succ 0$ for all $j=1,\ldots,m$. In this scenario, \eqref{probnewton} is equivalent to the convex quadratic optimization problem:
\begin{equation} \label{newtonproblem}
 \begin{array}{cl}
\ds\min_{(t,d)\in\R\times\R^n}    & \;  t          \\
\mbox{s. t.} & \nabla F_j (x)^\top d + \dfrac{1}{2}d^\top B_j d  \leq t,  \quad \forall j = 1,\dots,m.
\end{array}
\end{equation}
The unique solution to \eqref{newtonproblem} is given by $(t,d):=(\theta(x), d(x))$. Since \eqref{newtonproblem} is convex and has a Slater point (e.g., $(1, 0)\in\R\times\R^n$), there exists a multiplier $\lambda(x) \in \mathbb{R}^m$ such that the triple $(t,d,\lambda):=(\theta(x), d(x),\lambda(x)) \in\R\times\R^n\times\R^m$ satisfies its Karush-Kuhn-Tucker system given by:
\begin{equation}\label{kkt1}
\sum_{j=1}^{m} \lambda_j=1, \quad \sum_{j=1}^{m} \lambda_j\left[\nabla F_j (x) + B_jd\right]=0,
\end{equation}
\begin{equation}\label{kkt2}
\lambda_j\geq 0, \; \nabla F_j (x)^\top d + \frac{1}{2}d^\top B_j d \leq t, \quad \forall j=1,\dots,m,
\end{equation}
\begin{equation}\label{kkt3}
\lambda_j\left[\nabla F_j (x)^\top d + \frac{1}{2}d^\top B_j d-t\right]=0, \quad \forall j=1,\dots,m.
\end{equation}
In particular, \eqref{kkt1} and \eqref{kkt2} imply that
\begin{equation}\label{dN}
d(x) = - \left[ \sum_{j=1}^{m} \lambda_j(x)B_j \right]^{-1} \sum_{j=1}^{m}\lambda_j(x)\nabla F_j (x)
\end{equation}
and
\begin{equation}\label{lambdaN}
 \lambda(x) \in \Delta_m,
\end{equation}
where $\Delta_m$ represents the $m$-dimensional simplex defined as:
$$\Delta_m:=\{\lambda \in\R^m \mid \sum_{j=1}^m \lambda_j =1, \lambda \succeq 0\}.$$
  Now, by summing \eqref{kkt3} over $j=1,\ldots,m$ and using \eqref{dN}--\eqref{lambdaN}, we obtain
\begin{align}
  \theta(x)  & = \left[\sum_{j=1}^{m}\lambda_j\nabla F_j (x)\right]^\top d(x) + \frac{1}{2}d(x)^\top \left[ \sum_{j=1}^{m} \lambda_j(x)B_j \right] d(x) \nonumber \\
                  &=  - d(x)^\top \left[ \sum_{j=1}^{m} \lambda_j(x)B_j \right] d(x) + \frac{1}{2}d(x)^\top \left[ \sum_{j=1}^{m} \lambda_j(x)B_j \right] d(x) \nonumber \\
                  & = -\frac{1}{2} d(x)^\top \left[ \sum_{j=1}^{m} \lambda_j(x)B_j \right] d(x). \label{thetaN}
  \end{align}
  
\begin{lemma}\label{l:pareto} 
Let $d:\R^n\to\R^n$ and $\theta:\R^n\to\R$ given by \eqref{direction} and \eqref{theta}, respectively. Assume that $B_j\succ 0$ for all $j=1,\ldots,m$. Then, we have:
	\begin{itemize}
\item[(i)] $x$ is Pareto critical if and only if $d(x)=0$ and $\theta(x)=0$;
\item[(ii)] if $x$ is not Pareto critical, then $d(x)\neq0$ and $\D(x,d(x))< \theta(x)<0$ (in particular, $d(x)$ is a descent direction for $F$ at $x$).
	\end{itemize}
\end{lemma}
\begin{proof} 
See  \cite[Lemma 3.2]{mauriciobenarfliege} and  \cite[Lemma 2]{POVALEJ2014765}.
  \end{proof}

As previously mentioned, if $B_j=I_n$ for all $j=1,\ldots,m$, the solution of \eqref{probnewton} corresponds to the steepest descent direction, denoted by $\dsd(x)$:
 \begin{equation}{\label{directionsd}}
\dsd(x)\coloneqq\argmin_{d \in \R^n}\; \max_{j = 1,\dots,m }\nabla F_j (x)^\top d + \frac{1}{2}\|d\|^2.
\end{equation} 
Taking the above discussion into account, we can observe that there exists
\begin{equation}{\label{lsd}}
		\lambda^{SD}(x)\in\Delta_m
\end{equation}
such that
\begin{equation}{\label{dsd}}
\dsd(x) =- \ds\sum_{j=1}^{m} \lambda^{SD}_j(x)\nabla F_j (x).
\end{equation}
Next, we will review some useful properties related to $\dsd(\cdot)$.
\begin{lemma} \label{l:vtheta}
Let $d_{SD}:\R^n\to\R^n$ be given by \eqref{directionsd}. Then:
\begin{itemize}
\item[(i)] $x$ is Pareto critical if and only if $d_{SD}(x)=0$;
\item[(ii)] if $x$ is not Pareto critical, then we have $d_{SD}(x)\neq0$ and  $\D(x,d_{SD}(x))<-(1/2)\|d_{SD}(x)\|^2<0$ (in particular, $d_{SD}(x)$ is a descent direction for $F$ at $x$);
\item[(iii)] the mapping $d_{SD}(\cdot)$ is continuous;
\item[(iv)] for any $x\in\R^n$, $-d_{SD}(x)$ is the minimal norm element of the set 
$$\{u\in\R^n \mid u = \ds\sum_{j=1}^{m} \lambda_j\nabla F_j (x),\lambda\in\Delta_m \},$$
 i.e., in the convex hull of $\{\nabla F_1(x),\ldots,\nabla F_m(x)\}$;
 \item[(v)] if $\nabla F_j$, $j=1,\ldots,m$, are $L$-Lipschitz continuous on a nonempty set $U\subset \R^n$,  i.e.,
 $$\|\nabla F_j(x)-\nabla F_j(y)\|\leq L\|x-y\|, \quad \forall x,y\in U, \quad \forall j=1,\ldots,m,$$
 then the mapping $x\mapsto \|d_{SD}(x)\|$ is also $L$-Lipschitz continuous on $U$.
\end{itemize}
\end{lemma}
\begin{proof} 
For items {\it (i)}, {\it (ii)}, and {\it (iii)}, see~\cite[Lemma~3.3]{mauricio&benar}. For items {\it (iv)} and {\it (v)}, see \cite[Corollary~2.3]{Sbenar} and \cite[Theorem~3.1]{Sbenar}, respectively.
\end{proof}

We end this section by presenting an auxiliary result.

\begin{lemma} \label{l:natal}
The following statements are true.
\begin{itemize}
\item[(i)] The function $h(t):=1-t+\ln(t)$ is nonpositive for all $t>0$.
\item[(ii)] For any $\bar{t}<1$, we have $\ln(1-\bar{t})\geq -\bar{t}/(1-\bar{t})$.
\end{itemize}
\end{lemma}
\begin{proof}
For item~{\it (i)}, see \cite[Exercise 6.8]{NoceWrig06}. For item~{\it (ii)}, consider $\bar{t}<1$. By applying item~{\it (i)} with $t=1/(1-\bar{t})$, we obtain
\[0\geq h\Big(\dfrac{1}{1-\bar{t}}\Big)=1-\dfrac{1}{1-\bar{t}}+\ln\Big(\dfrac{1}{1-\bar{t}}\Big)=-\dfrac{\bar{t}}{1-\bar{t}}-\ln(1-\bar{t}).\]
\end{proof}
  
\section{The algorithm and its global convergence}   \label{sec:algorithm}

In this section, we define the main algorithm employed in this paper and study its global convergence, with a particular focus on nonconvex multiobjective optimization problems.
Let us suppose that the following usual assumptions are satisfied.

\begin{assumption}\label{assumption1} 
\textit{(i)} $F$ is continuously differentiable.\\
\textit{(ii)}  The level set $\Ls:=\{x\in\R^n\mid F(x) \; \preceq \; F(x^0)\}$ is bounded, where $x^0\in\R^n$ is the given starting point.\\
\textit{(iii)} There exists an open set ${\cal N}$ containing $\Ls$ such that $\nabla F_j$ is $L$-Lipschitz continuous on ${\cal N}$ for all $j = 1, \dots, m$, i.e.,
$$\|\nabla F_j(x)-\nabla F_j(y)\|\leq L \|x-y\|, \quad \forall x,y\in{\cal N}, \quad \forall j = 1,\ldots,m.$$
\end{assumption}

The algorithm is formally described as follows.

\noindent\rule{\textwidth}{1pt}\vspace{-10pt}
\begin{algorithm}\label{alg:MBFGS}
{ \normalfont \bf A BFGS-type algorithm for nonconvex problems}

	\begin{description}\normalfont
		Let $\rho\in(0,1/2)$, $\sigma\in(\rho,1)$, $0<\underline{\vartheta}\leq \bar{\vartheta}$, $x^0\in\mathbb{R}^n$, and $B_j^0\succ 0$ for all $j=1,\ldots,m$ be given. Initialize $k \leftarrow 0$.
		
		\item[Step 1.]  \textit{Compute the search direction}\\
		Compute $d^k \coloneqq d(x^k)$ as in \eqref{direction}. 
		
		\item[Step 2.] \textit{Stopping criterion}\\
		If $x^k$ is Pareto critical, then STOP.
		
		\item[Step 3.]  \textit{Line search procedure}\\
		Compute a step size $\alpha_{k}>0$ (trying first $\alpha_{k}=1$) such that
		\begin{align} 
			F_j(x^k+\alpha_k d^k) & \leq F_j(x^k)+\rho \alpha_k \D(x^k,d^k),  \quad \forall j = 1,\ldots,m, \label{wolfe1}  \\ 
	  	        \D(x^k+\alpha_kd^k,d^k) & \geq \sigma \D(x^k,d^k), \label{wolfe2}
		\end{align}
		and set $x^{k+1} \coloneqq x^k + \alpha_k d^k$.		
		
		\item[Step 4.] \textit{Prepare the next iteration}\\
		Compute
		\begin{equation}\label{etajk}
		\eta_j^k = \dfrac{(y_{j}^k)^\top s^k}{\|s^k\|^2}, \quad \forall j = 1,\ldots,m,
		\end{equation}
		where $s^k \coloneqq x^{k+1}-x^{k}$ and $y_j^k \coloneqq \nabla F_{j}(x^{k+1})-\nabla F_j(x^k)$. Choose $\mu^k\in\Delta_m$ and $\vartheta_k\in(\underline{\vartheta},\bar{\vartheta})$, and define
		\begin{equation}\label{rjk}
		r_j^k := \max \{-\eta_j^k, 0\} + \vartheta_k\|\sum_{i=1}^m\mu_i^k \nabla F_i(x^k)\|,  \quad \forall j = 1,\ldots,m,
		\end{equation}
		and
		\begin{equation}\label{gammajk}
		\gamma_j^k  \coloneqq y_j^k + r_j^ks^k,  \quad \forall j = 1,\ldots,m.
		\end{equation}
		
		\item[Step 5.]  \textit{Update the BFGS-type matrices}\\
		Define
		\begin{equation}\label{mBFGS}
			B_j^{k+1} \coloneqq B_j^{k} - \dfrac{B_j^{k}s^k(s^k)^\top B_j^{k}}{(s^k)^\top B_j^ks^k} +\dfrac{\gamma_j^k(\gamma_j^k)^\top }{(\gamma_j^{k})^\top s^k},  \quad \forall j = 1,\ldots,m.  \\
		\end{equation}
		Set $k\leftarrow k+1$ and go to Step~1.
		
	\end{description}
\vspace{-10pt}\noindent\rule{\textwidth}{1pt}
\end{algorithm}

Some comments are in order. 
First, by expressing the search direction subproblem \eqref{probnewton} as the convex quadratic optimization problem \eqref{newtonproblem}, we can apply well-established techniques and solvers to find its solution at Step~1.
 Second, some practical stopping criteria can be considered at Step~2. It is usual to use the {\it gap} function $\theta(x^k)$ in \eqref{theta} or the norm of $\dsd(x^k)$ in \eqref{directionsd} to measure criticality, see Lemmas~\ref{l:pareto} and \ref{l:vtheta}, respectively.
 Third, at Step 3, we require that $\alpha_k$ satisfies \eqref{wolfe1}--\eqref{wolfe2}, which corresponds to the multiobjective standard Wolfe conditions originally introduced in \cite{cgvo}.
 Under Assumption~\ref{assumption1}{\it(i)}--{\it(ii)}, given $d^k\in\R^n$ a descent direction for $F$ at $x^k$, it is possible to show that there are intervals of positive step sizes satisfying \eqref{wolfe1}--\eqref{wolfe2}, see \cite[Proposition~3.2]{cgvo}.
As we will see, under suitable assumptions, the unit step size $\alpha_k=1$  is eventually accepted, which is essential to obtain {\it fast} convergence.
Furthermore, an algorithm to calculate Wolfe step sizes for vector-valued problems was proposed in \cite{PerezPrudente2019}.
Fourth, the usual BFGS scheme for $F_j$ consists of the update formula given in \eqref{mBFGS} with $y_j^k$ in place of $\gamma_j^k$.
In this case, the product $(y_j^k)^\top s^k$ in the denominator of the third term on the right-hand side of \eqref{mBFGS} can be nonpositive for some $j\in\{1,\ldots,m\}$, even when the step size satisfies the Wolfe conditions~\eqref{wolfe1}--\eqref{wolfe2}, see \cite[Example~3.3]{Leandro&Danilo}.
This implies that update scheme (with $y_j^k$ in place of $\gamma_j^k$) may fail to preserve positive definiteness of $B_j^k$.
Fifth, note that $B_j^{k+1}s^k= y_j^k + r_j^ks^k$ for each $j=1,\ldots,m$. Thus, if $r_j^k$ is {\it small}, this relation can be seen as an approximation of the well-known secant equation $B_j^{k+1}s^k= y_j^k$ for $F_j$, see \cite{Li-Fu-MBFGS}.

\begin{theorem} \label{teo:well}
Algorithm~\ref{alg:MBFGS} is well-defined.
\end{theorem}
\begin{proof}
The proof is by induction. We start by assuming that $B_j^k\succ 0$ for all $j=1,\ldots,m$, which is trivially true for $k=0$. 
This makes subproblem \eqref{direction} in Step~1 solvable. 
If $x^k$ is Pareto critical, Algorithm~\ref{alg:MBFGS} stops at Step~2, thereby concluding the proof.
Otherwise, Lemma~\ref{l:pareto}{\it (ii)} implies that $d^k$ is a descent direction of $F$ at $x^k$. 
Thus, taking into account Assumption~\ref{assumption1}{\it(i)}--{\it(ii)}, there exist intervals of positive step sizes satisfying conditions \eqref{wolfe1}--\eqref{wolfe2}, as shown in \cite[Proposition~3.2]{cgvo}.
As a result, $x^{k+1}$ can be properly defined in Step~3. 
To complete the proof, let us show that $B_j^{k+1}$ remains positive definite for all $j=1,\ldots,m$.
By the definitions of $\gamma_j^{k}$ and $\eta_j^k$ in \eqref{gammajk} and \eqref{etajk}, respectively, we have
$$(\gamma_j^{k})^\top s^k =(y_j^k)^\top s^k + r_j^k \|s^k\|^2 = \left(\dfrac{(y_j^k)^\top s^k}{\|s^k\|^2} + r_j^k\right) \|s^k\|^2=(\eta_j^k+ r_j^k)  \|s^k\|^2, \quad \forall j=1,\ldots,m.$$
Therefore, by the definition of $r_j^k$ in \eqref{rjk}, Lemma~\ref{l:vtheta}{\it(iv)}, and Lemma~\ref{l:vtheta}{\it(ii)}, it follows that
\begin{equation}\label{curitiba}
(\gamma_j^{k})^\top s^k\geq \vartheta_k\|\sum_{i=1}^m\mu_i^k \nabla F_i(x^k)\| \|s^k\|^2\geq \underline{\vartheta} \|\dsd(x^k)\| \|s^k\|^2>0, \quad \forall j=1,\ldots,m.
\end{equation}
Thus, the updating formulas \eqref{mBFGS} are well-defined. Now, for each $j\in\{1,\ldots,m\}$ and any nonzero vector $z$, we have
\begin{equation}\label{eq:cs}
z^\top B_j^{k+1} z = \|z\|_{B_j^{k}}^2-\dfrac{\langle z,s^k \rangle_{B_j^{k}}^2}{\|s^k\|_{B_j^{k}}^2}+\dfrac{(z^\top\gamma_j^k)^2}{(\gamma_j^{k})^\top s^k} \geq \dfrac{(z^\top\gamma_j^k)^2}{(\gamma_j^{k})^\top s^k}\geq 0,
\end{equation}
where the first inequality is a direct consequence of the Cauchy-Schwarz inequality, which gives $\langle z,s^k \rangle_{B_j^{k}}^2\leq \|z\|_{B_j^{k}}^2 \|s^k\|_{B_j^{k}}^2$.
Finally, assume by contradiction that $z^\top B_j^{k+1} z=0$.
In this case, it follows from \eqref{eq:cs} that
\begin{equation}\label{eq:cs2}
z^\top\gamma_j^k=0 \quad \mbox{and} \quad  \|z\|_{B_j^{k}}^2-\dfrac{\langle z,s^k \rangle_{B_j^{k}}^2}{\|s^k\|_{B_j^{k}}^2}=0.
\end{equation}
The second equation in \eqref{eq:cs2} implies that $|\langle z,s^k \rangle_{B_j^{k}}|= \|z\|_{B_j^{k}} \|s^k\|_{B_j^{k}}$, so there exists $\tau\in\R$ such that $z=\tau s^k$.
Combining this with the first equation in \eqref{eq:cs2}, we obtain $\tau (\gamma_j^{k})^\top s^k=0$. Taking into account \eqref{curitiba}, we can deduce that $\tau=0$, which contradicts the fact that $z$ is a nonzero vector.
\end{proof}

Hereafter, we assume that $x^k$ is not Pareto critical for all $k\geq 0$. 
Thus, Algorithm~\ref{alg:MBFGS} generates an infinite sequence of iterates.
The following result establishes that the sequence $\{x^k,d^k\}$ satisfies a Zoutendijk-type condition, which will be crucial in our analysis. Its proof is based on \cite[Proposition~3.3]{cgvo} and will be provided in the Appendix.

\begin{proposition}\label{prop:zout}
Consider the sequence $\{x^k,d^k\}$ generated by Algorithm~\ref{alg:MBFGS}. Then,
\begin{equation}\label{Zout}
	\sum_{k\geq 0} \dfrac{\D(x^k,d^k)^2}{\|d^k\|^2} < \infty.
\end{equation}
\end{proposition}

Our analysis also exploits insights developed by Byrd and Nocedal  \cite{Tool} in their analysis of the classical BFGS method for single-valued optimization (i.e., for $m=1$).
They provided sufficient conditions that ensure that the angle between $s^k$ and $B_1^ks^k$ (which coincides with the angle between $d^k$ and $-\nabla F_1(x^k)$ in the scalar case)  remains far from $0$ for an arbitrary fraction of the iterates.
Recently, this result was studied in the multiobjective setting in \cite{Leandro&Danilo}.
Under some mild conditions, similar to the approach taken in \cite{Leandro&Danilo}, we establish that the angles between $s^k$ and $B_j^ks^k$ remain far from $0$, {\it simultaneously} for all objectives, for an arbitrary fraction $p$ of the iterates.
The proof of this result can be constructed as a combination of \cite[Theorem~2.1]{Tool} and \cite[Lemma~4.2]{Leandro&Danilo} and will therefore be postponed to the Appendix.

\begin{proposition}\label{l:poa}
Consider the sequence $\{x^k\}$ generated by Algorithm~\ref{alg:MBFGS}.
Let $\beta_j^k$ be the angle between the vectors $s^k$ and $B_j^ks^k$, for all $k\geq 0$ and $j=1,\ldots,m$.
Assume that
\begin{equation}\label{atletico}
\dfrac{(\gamma_j^{k})^\top s^k}{\|s^k\|^2}\geq C_1  \quad \mbox{and} \quad \dfrac{\|\gamma_j^{k}\|^2}{(\gamma_j^{k})^\top s^k}\leq C_2,  \quad \forall j=1,\ldots,m, \quad \forall k\geq 0,
\end{equation}
for some positive constants $C_1,C_2>0$.
Then, given $p\in(0,1)$, there exists a constant $\delta>0$ such that, for all $k \geq 1$, the relation
$$\cos \beta_j^{\ell} \geq \delta, \quad \forall j=1,\ldots,m,$$
holds for at least $\lceil p(k+1) \rceil$ values of $\ell \in \{0,1,\ldots,k\}$, where $\lceil \cdot \rceil$ denotes the ceiling function.
\end{proposition}

The following technical result forms the basis for applying Proposition~\ref{l:poa}.

\begin{lemma}\label{l:goiania}
Let $\{x^k\}$ be a sequence generated by Algorithm~\ref{alg:MBFGS}.
Then, for each j = 1,\ldots,m and all $k\geq 0$, there exist positive constants $c_1,c_2>0$ such that:
\begin{itemize}
\item[(i)] $\dfrac{(\gamma_j^{k})^\top s^k}{\|s^k\|^2}\geq c_1 \|\dsd(x^k)\|$;
\item[(ii)] $\dfrac{\|\gamma_j^{k}\|^2}{(\gamma_j^{k})^\top s^k}\leq \dfrac{c_2}{\|\dsd(x^k)\|}$.
\end{itemize}
\end{lemma}
\begin{proof}
Let $k\geq 0$ and $j\in\{1,\ldots,m\}$ be given. As in \eqref{curitiba}, we have
$$\dfrac{(\gamma_j^{k})^\top s^k}{\|s^k\|^2}  \geq \underline{\vartheta} \|\dsd(x^k)\|,  \quad \forall j=1,\ldots,m.$$
Thus, taking $c_1:=\underline{\vartheta}$, we conclude item~{\it(i)}.
Now consider item~{\it(ii)}. 
By the Cauchy--Schwarz inequality and Assumption~\ref{assumption1}{\it(iii)}, it follows that
$$|\eta_j^k|=  \dfrac{|(y_{j}^k)^\top s^k|}{\|s^k\|^2}\leq \dfrac{\|y_{j}^k\|}{\|s^k\|} = \dfrac{\|\nabla F_{j}(x^{k+1})-\nabla F_j(x^k)\|}{\|x^{k+1}-x^{k}\|}\leq L.$$
On the other hand, since $\{x^k\}\subset\Ls$ and $\mu^k\in\Delta_m$, by Assumption~\ref{assumption1}{\it(i)}--{\it(ii)}, there exists a constant $\bar{c}>0$, independent of $k$, such that $\|\sum_{i=1}^m\mu_i^k \nabla F_i(x^k)\|\leq \bar{c}$.
The definition of $r_j^k$ together with the last two inequalities yields 
$$r_j^k\leq |\eta_j^k| + \vartheta_k \|\sum_{i=1}^m\mu_i^k \nabla F_i(x^k)\| \leq L+ \bar{\vartheta}\bar{c},$$
and hence
$$\|\gamma_j^k\| \leq  \|y_j^k\| + r_j^k \|s^k\| = \left(\dfrac{\|y_j^k\|}{\|s^k\|} + r_j^k \right) \|s^k\| \leq (2L+ \bar{\vartheta}\bar{c}) \|s^k\|.$$
By squaring the latter inequality and using item{\it(i)}, we obtain
$$\dfrac{\|\gamma_j^{k}\|^2}{(\gamma_j^{k})^\top s^k}\leq (2L+ \bar{\vartheta}\bar{c})^2 \dfrac{ \|s^k\|^2}{(\gamma_j^{k})^\top s^k}  \leq \dfrac{(2L+ \bar{\vartheta}\bar{c})^2}{c_1 \|\dsd(x^k)\|}.$$
Thus, taking $c_2:=(2L+ \bar{\vartheta}\bar{c})^2/c_1$, we conclude the proof.
\end{proof}

From now on, let $\lambda^k:= \lambda(x^k)\in\R^m$ be the Lagrange multiplier associated with $x^k$  satisfying \eqref{dN}--\eqref{lambdaN}.
We are now able to prove the main result of this section. 
We show that Algorithm~\ref{alg:MBFGS} finds a Pareto critical point of $F$, without imposing any convexity assumptions.

\begin{theorem}\label{bh}
Let $\{x^k\}$ be a sequence generated by Algorithm~\ref{alg:MBFGS}. Then
\begin{equation}\label{saopaulo}
\liminf_{k\to\infty}\|\dsd(x^k)\|=0.
\end{equation}
As a consequence, $\{x^k\}$ has a limit point that is Pareto critical.
\end{theorem}
\begin{proof}
Assume by contradiction that there is a constant $\varepsilon>0$ such that
\begin{equation}\label{floripa}
\|\dsd(x^k)\| \geq \varepsilon, \quad \forall k\geq 0.
\end{equation}
From Lemma~\ref{l:goiania}, taking $C_1:=c_1\varepsilon$ and $C_2:=c_2/\varepsilon$, we have
$$\dfrac{(\gamma_j^{k})^\top s^k}{\|s^k\|^2}\geq C_1  \quad \mbox{and} \quad \dfrac{\|\gamma_j^{k}\|^2}{(\gamma_j^{k})^\top s^k}\leq C_2,  \quad \forall j=1,\ldots,m, \quad \forall k\geq 0,$$
showing that the assumptions of Proposition~\ref{l:poa} are satisfied. Thus, there exist a constant $\delta>0$ and $\K\subsetinf\N$ such that 
\[\cos\beta_j^{k} \geq \delta, \quad  \forall j = 1,\ldots,m, \quad \forall k\in \K.\] 
Hence, by the definitions of $\cos\beta_j^k$ and $s^k$, we have, for all $j = 1,\ldots,m$,
\[\delta\leq  \cos\beta_j^k = \frac{(s^k)^\top B_j^ks^k}{\|s^k\|\|B_j^ks^k\|}=\frac{(d^k)^\top B_j^kd^k}{\|d^k\|\|B_j^kd^k\|}, \quad \forall k\in \K,\]
which implies
\[(d^k)^\top B_j^kd^k \geq \delta \|d^k\|\|B_j^kd^k\|, \quad  \forall k\in \K.\]
Therefore, from Lemma~\ref{l:pareto}{\it (ii)} and \eqref{thetaN}, it follows that
\[-\D(x^k,d^k)>-\theta(x^k) =  \frac{1}{2} \sum_{j=1}^{m} \lambda_j^k (d^k)^\top B_j^k d^k \geq \frac{\delta}{2}  \|d^k\| \sum_{j=1}^{m} \lambda_j^k \|B_j^kd^k\|, \quad \forall k\in \K.\]
Thus, from the triangle inequality, \eqref{dN}, \eqref{lambdaN}, Lemma~\ref{l:vtheta}{\it (iv)}, and \eqref{floripa}, we obtain
\[-\dfrac{\D(x^k,d^k)}{\|d^k\|} \geq \frac{ \delta}{2}  \|\sum_{j=1}^{m} \lambda_j^k B_j^kd^k\| =  \frac{\delta}{2}  \|\sum_{j=1}^{m} \lambda_j^k \nabla F_j(x^k)\|\geq  \frac{ \delta}{2}  \|\dsd(x^k)\| \geq \frac{ \delta \varepsilon}{2}, \quad \forall k\in \K.\]
Hence, 
\[\sum_{k\geq 0}  \dfrac{\D(x^k,d^k)^2}{\|d^k\|^2}\geq  \sum_{k\in \K} \dfrac{\D(x^k,d^k)^2}{\|d^k\|^2}\geq \sum_{k\in \K} \dfrac{\delta^2\varepsilon^2}{4}=\infty,\]
which contradicts the Zoutendijk condition \eqref{Zout}. Therefore, we conclude that \eqref{saopaulo} holds.

Now, \eqref{saopaulo} implies that there exists $\K_1\subsetinf\N$ such that $\lim_{k\in\K_1}\|\dsd(x^k)\|=0$. On the other hand, given that $\{x^k\}\subset \Ls$ and $\Ls$ is compact, we can establish the existence of $\K_2\subseteq \K_1$ and  $x^{\ast}\in \Ls$ such that $\lim_{k\in \K_2}x^k= x^{\ast}$. Thus, from Lemma~\ref{l:vtheta}{\it (iii)}, we deduce that $\dsd(x^{\ast})=0$. Consequently, based on Lemma~\ref{l:vtheta}{\it (i)}, we conclude that $x^{\ast}$ is Pareto critical.
\end{proof}

Even though the primary focus of this article is on nonconvex problems, we conclude this section by establishing full convergence of the sequence generated by Algorithm~\ref{alg:MBFGS} in the case of strict convexity of $F$. Note that, under Assumption~\ref{assumption1}{\it (i)--(ii)}, the existence of at least one Pareto optimal point is assured in this particular case.

\begin{theorem}
Let $\{x^k\}$ be a sequence generated by Algorithm~\ref{alg:MBFGS}. If $F$ is strictly convex, then $\{x^k\}$ converges to a Pareto optimal point of $F$.
\end{theorem}
\begin{proof}
According to Theorem~\ref{bh} and Theorem~\ref{Theo3.1-Bennar-Fliege}{\it (iii)}, there exists a limit point $\xs\in\Ls$ of $\{x^k\}$ that is Pareto optimal.
Let $\K_1\subsetinf \N$ be such that  $\lim_{k\in \K_1}x^k=\xs$.
To show the convergence of $\{x^k\}$ to $\xs$, let us suppose by contradiction that there exist $\bar{x}\in\mathcal{L}$, where $\bar{x}\neq x^{\ast}$, and $\K_2\subsetinf \N$ such that $\lim_{k\in \K_2}x^k=\bar{x}$.
We first claim that $F(\bar{x})\neq F(x^{\ast})$. In fact, if $F(\bar{x})= F(x^{\ast})$, based on \eqref{eq:convex}, for all $t\in(0,1)$, we would have
\[F((1-t)x^{\ast}+t\bar{x})\prec (1-t)F(x^{\ast})+tF(\bar{x})=F(x^{\ast}),\]
which contradicts the fact that $x^{\ast}$ is a Pareto optimal point. Hence, $F(\bar{x})\neq F(x^{\ast})$, as we claimed.
Now, since $x^{\ast}$ is Pareto optimal, there exists $j_\ast\in\{1,\ldots,m\}$ such that $F_{j_\ast}(x^{\ast})<F_{j_\ast}(\bar{x})$. Therefore, considering that $\lim_{k\in \K_1}x^k=x^{\ast}$ and $\lim_{k\in \K_2}x^k=\bar{x}$, we can choose $k_1\in \K_1$ and $k_2\in \K_2$ such that $k_1<k_2$ and $F_{j_\ast}(x^{k_1})<F_{j_\ast}(x^{k_2})$. This contradicts \eqref{wolfe1} which implies, in particular, that $\{F_{j_\ast}(x^k)\}$ is decreasing. Thus, we can conclude that $\lim_{k\to \infty}x^k=x^{\ast}$, completing the proof.
\end{proof}

\section{Local convergence analysis} \label{sec:superlinear}

In this section, we analyze the local convergence properties of Algorithm~\ref{alg:MBFGS}. 
The findings presented here are applicable to both convex and nonconvex problems. 
We will assume that the sequence $\{x^k\}$ converges to a local Pareto optimal point $x^{\ast}$ and show, under appropriate assumptions, that the convergence rate is superlinear.
\subsection{Superlinear rate of convergence}
Throughout this section, we make the following assumptions.

\begin{assumption}\label{assumption2}
\textit{(i)} $F$ is twice continuously differentiable.\\
\textit{(ii)}  The sequence $\{x^k\}$ generated by Algorithm~\ref{alg:MBFGS} converges to a local Pareto optimal point $x^{\ast}$ where $\HF_j(\xs)$ is positive definite for all $j=1,\ldots,m$.\\
\textit{(iii)} For each $j=1,\ldots,m$, $\HF_j(x)$ is Hölder continuous at $\xs$, i.e., there exist constants $\nu\in(0,1]$ and $M>0$ such that
 \begin{equation}\label{lip}
	\|\HF_j(x)-\HF_j(\xs)\|\leq M \|x-\xs\|^{\nu}, \quad \forall j=1,\ldots,m,
\end{equation} 
for all $x$ in a neighborhood of $\xs$.
\end{assumption}

Under Assumption~\ref{assumption2}{\it (ii)}, there exist a neighborhood $U$ of $\xs$ and constants $\underline{L}>0$ and $L>0$ such that
 \begin{equation}\label{sconvex}
	 \underline{L}\|z\|^2 \leq z^\top \HF_j(x)z\leq L\|z\|^2, \quad \forall j=1,\ldots,m,
\end{equation} 
for all $z \in  \mathbb{R}^n$ and $x \in U$. 
In particular, \eqref{sconvex} implies that $F_j$ is strongly convex and has Lipschitz continuous gradients on $U$.
Note that constant $L$ in \eqref{sconvex} aligns with the $L$ defined in Assumption~\ref{assumption1}{\it (iii)} as part of the $L$-Lipschitz continuity condition for $\nabla F_j$, maintaining consistent notation for the Lipschitz constant across our analysis. Throughout this section, we assume, without loss of generality, that $\{x^k\}\subset U$ and that Assumption~\ref{assumption2}{\it (iii)} holds in $U$, i.e., \eqref{lip}  and \eqref{sconvex} hold at $x^k$ for all $k\geq 0$.

We also introduce the following additional assumption about $\{r_j^k\}$, which will be considered only when explicitly mentioned. In Section~\ref{sec:rjk}, we will explore practical choices for $\{r_j^k\}$ that satisfy such an assumption.

\begin{assumption}\label{assumption3}
For each $j=1,\ldots,m$, $\{r_j^k\}$ satisfies $\sum_{k \geq 0} r_j^k < \infty.$
\end{assumption}

The following result, which is related to the linear convergence of the sequence $\{x^k\}$ and is based on \cite[Theorem~4.6]{Leandro&Danilo}, has its proof in the Appendix.

\begin{proposition}\label{prop:linear}
Suppose that Assumption~\ref{assumption2}{\it (i)--(ii)} holds. Let $\{x^k\}$ be a sequence generated by Algorithm~\ref{alg:MBFGS}. Then, for all $\nu>0$, we have
\begin{equation}\label{argentina}
\sum_{k\geq 0}\|x^k-x^{\ast}\|^{\nu} < \infty.
\end{equation}
\end{proposition}

As usual in quasi-Newton methods, our analysis relies on the Dennis--Mor\'e \cite{DennisMore} characterization of superlinear convergence.
To accomplish this, we use a set of tools developed in \cite{Tool} (see also \cite{powell1976some}). For every $k\geq 0$, we define the {\it average Hessian} as:
	\[ \bar{G}_j^k \coloneqq \int_{0}^{1} \HF_j(x^k + \tau s^k)d\tau,\quad \forall j = 1,\dots,m.\]
	This leads to the relationship:
	\begin{equation}\label{uruguai}
		\bar{G}_j^ks^k = y_j^k, \quad \forall j = 1,\dots,m.
	\end{equation}
We also introduce, for each $j=1,\ldots,m$ and $k\geq 0$, the following quantities:
\[\ts := \nabla^2F_j(x^{*})^{1/2}s^k, \quad \ty  := \nabla^2F_j(x^{*})^{-1/2} y_j^{k}, \quad \tg:= \nabla^2F_j(x^{*})^{-1/2} \gamma_j^k,\]
and
\[\tB := \nabla^2F_j(x^{*})^{-1/2} B_j^{k}\nabla^2F_j(x^{*})^{-1/2}.\]
Note that
\[\tilde{B}_j^{k+1} = \tB- \dfrac{\tB\ts(\ts)^\top\tB}{(\ts)^\top\tB\ts} + \dfrac{\tg(\tg)^\top}{(\tg)^\top\ts}, \quad \forall j=1,\ldots,m,\]
and
\begin{equation}\label{joaopessoa}
\dfrac{(\tg)^\top\ts}{\|s^k\|^2}=\dfrac{(\gamma_j^k)^\top s^k}{\|s^k\|^2}=\dfrac{(y_j^k)^\top s^k}{\|s^k\|^2}+r_j^k= \dfrac{(s^k)^\top \bar{G}_j^ks^k}{\|s^k\|^2}+r_j^k\geq \underline{L}+r_j^k\geq \underline{L} , \quad \forall j=1,\ldots,m,
\end{equation}
  where the first inequality follows from the left hand side of \eqref{sconvex}. Considering that $\tB \succ 0$ and, based on \eqref{joaopessoa}, it follows that $(\tg)^\top\ts>0$, we can follow the same arguments as in the proof of Theorem~\ref{teo:well} to show that $\tilde{B}_j^{k+1}\succ 0$ for all $j=1,\ldots,m$ and all $k\geq 0$.
In connection with Proposition~\ref{l:poa} and Lemma~\ref{l:goiania}, we additionally define the following quantities:
\[\ta:= \dfrac{(\tg)^\top\ts}{\|\ts\|^2}, \quad \tb:= \dfrac{\|\tg\|^2}{(\tg)^\top\ts}, \quad \cos \tilde{\beta}_j^k  := \dfrac{(\ts)^\top\tB\ts}{\|\ts\|\|\tB \ts\|}, \quad \mbox{and} \quad \tq := \dfrac{(\ts)^\top\tB\ts}{\|\ts\|^2}.\]

Another useful tool combines the trace and the determinant of a given positive definite matrix $B$, through the following function:
\begin{equation}\label{trdet}
\psi(B):=\tr(B)-\ln(\det(B)).
\end{equation}
Given that, for all $j = 1,\ldots,m$,
$$\tr(\tilde{B}_j^{k+1})=\tr(\tB)-\dfrac{\|\tB\ts\|^2}{(\ts)^\top\tB\ts}+\dfrac{\|\tg\|^2 }{(\tg)^\top\ts}= \tr(\tB)-\dfrac{\tq}{\cos^2\tilde{\beta}_j^k}+\tb$$
and
$$\det(\tilde{B}_j^{k+1})=\det(\tB)\dfrac{(\tg)^\top\ts}{(\ts)^\top\tB\ts}=\det(\tB)\dfrac{\ta}{\tq},$$
we can perform some algebraic manipulations to obtain:
 \begin{equation}\label{fortaleza}
 \psi(\tilde{B}_j^{k+1}) = \psi(\tB) +\left(\tb - \ln(\ta) -1\right) +\left[1 - \dfrac{\tq}{\cos^2\tilde{\beta}_j^k} + \ln\bigg(\dfrac{\tq}{\cos^2\tilde{\beta}_j^k}\bigg)\right]+ \ln(\cos^2\tilde{\beta}_j^k).
   \end{equation}

We are now ready to prove the central result of the superlinear convergence analysis: We establish that the Dennis--Mor\'e condition  holds individually for each objective function $F_j$.
A similar result in the scalar case was given in \cite[Theorem~3.8]{Li-Fu-MBFGS}.

\begin{theorem} \label{teo:saoluis}
Suppose that Assumptions~\ref{assumption2} and \ref{assumption3} hold. Let $\{x^k\}$ be a sequence generated by Algorithm~\ref{alg:MBFGS}. Then, for each $j=1,\ldots,m$, we have
\begin{equation}\label{limq}
	\lim_{k\to\infty} \dfrac{(s^k)^\top B_j^ks^k}{(s^k)^\top \HF_j(\xs)s^k} = 1,
\end{equation}
and
\begin{equation}\label{aracaju}
 \lim_{k \to\infty}\dfrac{\| (B_j^k - \nabla ^2F_j(\xs))d^k\|}{\|d^k\|} = 0.
  \end{equation}
\end{theorem}
\begin{proof}
 Let $j\in\{1,\ldots, m\}$ be an arbitrary index. From \eqref{uruguai}, we obtain
$$y_j^k-\nabla^2F_j(x^{*})s^k=[\bar{G}_j^k-\nabla^2F_j(x^{*})]s^k,$$
and hence
$$\ty-\ts=\nabla^2F_j(x^{*})^{-1/2}[\bar{G}_j^k-\nabla^2F_j(x^{*})]\nabla^2F_j(x^{*})^{-1/2}\ts,$$
for all $k\geq 0$. Therefore, by the definition of $\bar{G}_j^k$ and \eqref{lip}, we obtain
\begin{align}
\|\ty-\ts\|\leq & \; \|\nabla^2F_j(x^{*})^{-1/2}\|^2 \|\ts\|  \|\bar{G}_j^k-\nabla^2F_j(x^{*})\|  \nonumber \\
   \leq  & \; M  \|\nabla^2F_j(x^{*})^{-1/2}\|^2  \|\ts\|   \int_{0}^{1}\| x^k + \tau s^k - x^{*}\|^\nu d\tau \leq  \bar{c}_j \varepsilon_k \|\ts\|,  \quad \forall k\geq 0,\label{brasilia}
\end{align}
where $\bar{c}_j:=M \|\nabla^2F_j(x^{*})^{-1/2}\|^2$ and $\varepsilon_k:=\max\{\|x^{k+1}-\xs\|^{\nu},\|x^{k}-\xs\|^{\nu}\}$.
Now, since $|\|\ty\|-\|\ts\||\leq\|\ty-\ts\|$, it follows from \eqref{brasilia} that
 \begin{equation}\label{vitoria}
 (1-\bar{c}_j\varepsilon_k)\|\ts\|\leq \|\ty\| \leq  (1+\bar{c}_j\varepsilon_k)\|\ts\|.
 \end{equation}
Without loss of generality, let us assume that $1-\bar{c}_j\varepsilon_k>0$, for all $k\geq 0$.
Therefore, the left hand side of \eqref{vitoria} together with \eqref{brasilia} yields
$$ (1-\bar{c}_j\varepsilon_k)^2\|\ts\|^2-2(\ty)^\top\ts+\|\ts\|^2 \leq \|\ty\|^2 -2(\ty)^\top\ts+\|\ts\|^2 \leq \bar{c}_j^2 \varepsilon_k^2 \|\ts\|^2,$$
so that
 \begin{equation}\label{maceio} 
 2(\ty)^\top\ts\geq (1-\bar{c}_j\varepsilon_k)^2\|\ts\|^2+\|\ts\|^2-\bar{c}_j^2 \varepsilon_k^2 \|\ts\|^2=2(1-\bar{c}_j\varepsilon_k) \|\ts\|^2.
  \end{equation}
By the definition of $\ta$, we have
\[\ta=\dfrac{(\ty+r_j^k \nabla^2F_j(x^{*})^{-1}\ts)^\top\ts}{\|\ts\|^2}=\dfrac{(\ty)^\top\ts}{\|\ts\|^2}+r_j^k\dfrac{(\ts)^\top\nabla^2F_j(x^{*})^{-1}\ts}{\|\ts\|^2}.\]
 Thus, by \eqref{sconvex} and \eqref{maceio}, we obtain
 \begin{equation}\label{salvador} 
\ta\geq 1-\bar{c}_j\varepsilon_k + \dfrac{r_j^k}{L}\geq 1-\bar{c}_j\varepsilon_k.
  \end{equation}
From the definition of $\tb$, \eqref{joaopessoa}, the right hand side of \eqref{vitoria}, and \eqref{salvador}, by performing some manipulations, we also obtain 
  \begin{align}
 \tb \; = \;&  \dfrac{\|\ty\|^2}{\ta\|\ts\|^2}+2r_j^k\dfrac{(\ty)^\top\nabla^2F_j(x^{*})^{-1}\ts}{(\tg)^\top\ts}+(r_j^k)^2\dfrac{(\ts)^\top\nabla^2F_j(x^{*})^{-2}\ts}{(\tg)^\top\ts} \nonumber \\
  \leq & \; \dfrac{(1+\bar{c}_j\varepsilon_k)^2}{1-\bar{c}_j\varepsilon_k} + 2r_j^k\dfrac{\|\nabla^2F_j(x^{*})^{-1/2}\|\|y_j^{k}\| \|\nabla^2F_j(x^{*})^{-1}\| \|\nabla^2F_j(x^{*})^{1/2}\| \|s^k\|}{\underline{L}\|s^k\|^2} \nonumber\\
  & \; +  (r_j^k)^2\dfrac{\|\nabla^2F_j(x^{*})^{-2}\| \|\ts\|^2}{\underline{L}\|s^k\|^2}  \nonumber \\
    \leq & \; 1+ \dfrac{3\bar{c}_j+\bar{c}_j^2\varepsilon_k}{1-\bar{c}_j\varepsilon_k}\varepsilon_k
    + \dfrac{2LC_1}{\underline{L}}r_j^k +  \dfrac{C_2}{\underline{L}}(r_j^k)^2, \label{teresina}
 \end{align}
 where $C_1:=\|\nabla^2F_j(x^{*})^{-1/2}\| \|\nabla^2F_j(x^{*})^{-1}\| \|\nabla^2F_j(x^{*})^{1/2}\|$ and $C_2:=\|\nabla^2F_j(x^{*})^{-2}\| \|\nabla^2F_j(x^{*})\|$. Assumptions~\ref{assumption2}{\it (ii)} and \ref{assumption3} imply that $\varepsilon_k \to 0$ and $r_j^k\to 0$, respectively. Thus, by using \eqref{salvador} and \eqref{teresina}, for all sufficiently large $k$, we have $\bar{c}_j\varepsilon_k <1/2$, $(r_j^k)^2\leq r_j^k$, and there exists a constant $C>\max\{3\bar{c}_j,(2LC_1+C_2)/\underline{L}\}$ such that
\[\ln(\ta)\geq \ln(1-\bar{c}_j\varepsilon_k)\geq -\dfrac{\bar{c}_j\varepsilon_k}{1-\bar{c}_j\varepsilon_k}\geq -2\bar{c}_j\varepsilon_k>-2C\varepsilon_k,\]
where the second inequality follows from Lemma~\ref{l:natal}~{\it (ii)} (with $\bar{t}=\bar{c}_j\varepsilon_k$), and 
\[\tb\leq 1+ C\varepsilon_k + C r_j^k.\]
 Let $k_0\in\N$ be such that the latter two inequalities hold for all $k\geq k_0$.
 Therefore, by \eqref{fortaleza}, we obtain
 \[\ln\Big(\dfrac{1}{\cos^2\tilde{\beta}_j^k}\Big)- \left[1 - \dfrac{\tq}{\cos^2\tilde{\beta}_j^k} + \ln\bigg(\dfrac{\tq}{\cos^2\tilde{\beta}_j^k}\bigg)\right]<  \psi(\tB)-\psi(\tilde{B}_j^{k+1}) + 3C\varepsilon_k + C r_j^k,\]
  for all $k\geq k_0$. By summing this expression and making use of \eqref{argentina} and Assumption~\ref{assumption3}, we have   
 \[\sum_{\ell\geq k_0} \left\{\ln\Big(\dfrac{1}{\cos^2\tilde{\beta}_j^\ell}\Big)- \left[1 - \dfrac{\tilde{q}_j^\ell}{\cos^2\tilde{\beta}_j^\ell} + \ln\bigg(\dfrac{\tilde{q}_j^\ell}{\cos^2\tilde{\beta}_j^\ell}\bigg)\right]\right\}\leq  \psi(\tilde{B}_j^{k_0}) + 3C\sum_{\ell\geq k_0} \varepsilon_k + C \sum_{\ell\geq k_0} r_j^\ell<\infty.\]
Since $\ln(1/\cos^2\tilde{\beta}_j^\ell)>0$ for all $\ell\geq k_0$ and, by Lemma~\ref{l:natal}~{\it (i)}, the term in the square brackets is nonpositive, we have
\[\lim_{\ell\to\infty}\ln\Big(\dfrac{1}{\cos^2\tilde{\beta}_j^\ell}\Big)=0 \quad \mbox{and} \quad\lim_{\ell\to\infty} \left[1 - \dfrac{\tilde{q}_j^\ell}{\cos^2\tilde{\beta}_j^\ell} + \ln\bigg(\dfrac{\tilde{q}_j^\ell}{\cos^2\tilde{\beta}_j^\ell}\bigg)\right] = 0,\]
and hence
 \begin{equation}\label{palmas} 
 \lim_{\ell\to\infty} \cos^2\tilde{\beta}_j^\ell =1 \quad \mbox{and} \quad\lim_{\ell\to\infty} \tilde{q}_j^\ell = 1.
   \end{equation}
Note that the second limit in \eqref{palmas} is equivalent to \eqref{limq}. Now, it follows that
\begin{align*}
\lim_{k\to\infty}\dfrac{\|\nabla^2F_j(x^{*})^{-1/2}(B_j^k - \nabla^2F_j(x^{*}))s^k\|^2}{\|\nabla^2F_j(x^{*})^{1/2}s^k\|^2} 
& = \lim_{k\to\infty} \dfrac{\|(\tB - I_n)\ts\|^2}{\|\ts\|^2} \\
& = \lim_{k\to\infty}\dfrac{\|\tB\ts\|^2 - 2(\ts)^\top\tB\ts + \|\ts\|^2}{\|\ts\|^2} \\
& = \lim_{k\to\infty}\left[\dfrac{(\tq)^2}{\cos^2\tilde{\beta}_j^k} - 2\tq + 1 \right] =0,
\end{align*} 
where the last equality follows from \eqref{palmas}.
The above limit trivially implies \eqref{aracaju}, concluding the proof.
\end{proof}

Based on the Dennis--Mor\'e characterization established in Theorem~\ref{teo:saoluis}, we can easily replicate the proofs presented in \cite[Theorem~5.5]{Leandro&Danilo} and \cite[Theorem~5.7]{Leandro&Danilo} to show that unit step size eventually satisfies the Wolfe conditions~\eqref{wolfe1}--\eqref{wolfe2} and the rate of convergence is superlinear, as detailed in the Appendix. We formally state the results as follows.

\begin{theorem}\label{teo:convergence}
Suppose that Assumptions~\ref{assumption2} and \ref{assumption3} hold. Let $\{x^k\}$ be a sequence generated by Algorithm~\ref{alg:MBFGS}. Then, the step size $\alpha_k =1$ is admissible for all sufficiently large $k$ and $\{x^k\}$ converges to $\xs$ at a superlinear rate.
\end{theorem}

\subsection{Suitable choices for $r_j^k$} \label{sec:rjk}

As we have seen, Algorithm~\ref{alg:MBFGS} is globally convergent regardless of the particular choice of $r_j^k$. On the other hand, the superlinear convergence rate depends on whether $r_j^k$ satisfies Assumption~\ref{assumption3}. Next, we explore suitable choices for the multiplier $\mu^k\in\Delta_m$ in \eqref{rjk} to ensure that $r_j^k$ satisfies the aforementioned assumption. 
In what follows, we will assume that Assumption~\ref{assumption2} holds. First note that, as in \eqref{joaopessoa}, we have $\eta_j^k = (y_{j}^k)^\top s^k/\|s^k\|^2\geq \underline{L}>0$ and hence
\[r_j^k =\max \{-\eta_j^k, 0\} + \vartheta_k\|\sum_{i=1}^m\mu_i^k \nabla F_i(x^k)\| =  \vartheta_k\|\sum_{i=1}^m\mu_i^k \nabla F_i(x^k)\|,  \quad \forall j=1,\ldots,m, \quad \forall k\geq 0.\]

\noindent {\bf Choice 1}: One natural choice is to set $\mu^k:=\lambda^{SD}(x^k)\in\Delta_m$ for all $k\geq 0$, where $\lambda^{SD}(x^k)$ is the steepest descent Lagrange multiplier associated with $x^k$ as in \eqref{dsd}. In this case, by Lemma~\ref{l:vtheta}{\it (i)} and Lemma~\ref{l:vtheta}{\it (v)}, we have
\[r_j^k = \vartheta_k\|\dsd(x^k)\| = \vartheta_k (\|\dsd(x^k)\| - \|\dsd(\xs)\|)\leq \bar{\vartheta} L \|x^k - \xs \|,  \quad \forall j=1,\ldots,m, \quad \forall k\geq 0.\]
By summing this expression and making use of \eqref{argentina}, we conclude that $r_j^k$ satisfies Assumption~\ref{assumption3} for all $j=1,\ldots,m$. One potential drawback of this approach is the need to compute the multipliers $\lambda^{SD}(x^k)$, which involves solving the subproblem in \eqref{directionsd}.

\vspace{8pt}
\noindent {\bf Choice 2}:  Another natural choice is to set $\mu^k:=\lambda^k \in\Delta_m$ for all $k\geq 0$, where $\lambda^k$ is the Lagrange multiplier corresponding to the search direction $d^k$, see \eqref{dN}.
Since the subproblem in Step~2 is typically solved in the form of \eqref{newtonproblem} using a primal-dual algorithm, this approach does not require any additional computational cost.
Let us assume that the sequences $\{B_j^k\}$ and $\{(B_j^k)^{-1}\}$ are bounded for all $j=1,\ldots,m$.
In this case, using \cite[Lemma~6]{gonccalves2021globally}, there exists a constant $\delta>0$ such that
$\|\sum_{i=1}^m\lambda_i^k \nabla F_i(x^k)\|\leq \delta \|\dsd(x^k)\|$ for all $k\geq 0$. Therefore, for all $j=1,\ldots,m$ and $k\geq 0$, similarly to the previous choice, we have
\[r_j^k = \vartheta_k\|\sum_{i=1}^m\lambda_i^k \nabla F_i(x^k)\| \leq \delta \bar{\vartheta}  \|\dsd(x^k)\|= \delta\bar{\vartheta} (\|\dsd(x^k)\| - \|\dsd(\xs)\|)\leq \delta\bar{\vartheta} L \|x^k - \xs \|,\]
and hence $r_j^k$ satisfies Assumption~\ref{assumption3} for all $j=1,\ldots,m$.

\section{Numerical experiments} \label{sec:numerical}

In this section, we present some numerical experiments to evaluate the effectiveness of the proposed scheme.
We are particularly interested in verifying how the introduced modifications affect the numerical performance of the method.
Toward this goal, we considered the following methods in our tests.
\begin{itemize}

\item Algorithm~\ref{alg:MBFGS} (Global BFGS): our globally convergent algorithm with $r_j^k$ chosen according to Choice~2 (see Section~\ref{sec:rjk}) and $\vartheta_k=0.1$ for all $k\geq 0$. It is worth noting that preliminary numerical tests demonstrated the superior efficiency of Choice~2 over Choice~1.

\item BFGS-Wolfe \cite{Leandro&Danilo}:  a BFGS algorithm  in which the Hessian approximations are updated, for each $j=1,\ldots,m$, by
\begin{align*}
\begin{split}
B_j^{k+1} \coloneqq &  B_j^{k} - \dfrac{(\rho_j^k)^{-1}B_j^{k}s^k(s^k)^\top B_j^{k}+[(s^k)^\top B_j^{k}s^k] y_j^{k}			(y_j^{k})^\top }{\left[(\rho_j^k)^{-1}-(y_j^{k})^\top s^k\right]^2 + (\rho_j^k)^{-1}(s^k)^\top B_j^ks^k}  \\
			& + \left[(\rho_j^k)^{-1} - (y_j^{k})^\top s^k\right]\dfrac{y_j^{k}(s^k)^\top B_j^{k} + B_j^{k}s^k(y_j^{k})^\top }{\left[(\rho_j^k)^{-1}-(y_j^{k})^\top s^k\right]^2 + (\rho_j^k)^{-1}(s^k)^\top B_j^ks^k}, 
			\end{split}
		\end{align*}
		where 	
\begin{equation*}\label{rho}
		\rho_j^k \coloneqq \left\{\begin{array}{ll}
		1/\left( (y_j^{k})^\top s^k\right), & \mbox{if } (y_j^{k})^\top s^k > 0\\
		1/\left(\D(x^{k+1},s^k)-\nabla F_j(x^k)^\top s^k\right), & \mbox{otherwise}.\\
		\end{array}\right.
		\end{equation*} 
and the step sizes are calculated satisfying the Wolfe conditions \eqref{wolfe1}--\eqref{wolfe2}.
We point out that this algorithm is well-defined for nonconvex problems, although it is not possible to establish global convergence in this general case. Additionally, in the case of scalar optimization ($m = 1$), it retrieves the classical scalar BFGS algorithm.

\item Cautious BFGS-Armijo \cite{QU2011397}: a BFGS algorithm  in which the Hessian approximations are updated, for each $j=1,\ldots,m$, by
\begin{equation*}\label{StBFGS}
B_j^{k+1}\!\coloneqq\!\left\{\hspace{-5pt}
\begin{tabular}{ll}
 $B_j^{k}\! -\! \dfrac{B_j^{k}s^k(s^k)^\top B_j^{k}}{(s^k)^\top B_j^ks^k}\!  +\!\dfrac{y_j^{k}(y_j^{k})^\top }{(y_j^{k})^\top s^k}$, &\hspace{-12pt} if $(y_j^{k})^\top s^k\geq \varepsilon \min\{1,|\theta(x^k)|\}$,\\
 $B_j^{k}$, & \hspace{-12pt} otherwise,
\end{tabular}\right.
\end{equation*}	
where $\varepsilon>0$ is an algorithmic parameter and the step sizes are calculated satisfying the Armijo-type condition given in \eqref{wolfe1}. In our experiments, we set $\varepsilon=10^{-6}$. This combination also leads to a globally convergent scheme, see \cite{QU2011397}.
\end{itemize}

We implemented the algorithms using Fortran~90. The search directions $d(x^k)$ (see \eqref {direction}) and optimal values $\theta(x^k)$ (see \eqref{theta}) were obtained by solving subproblem \eqref{newtonproblem} using the software Algencan \cite{algencan}. To compute step sizes satisfying the Wolfe conditions \eqref{wolfe1}--\eqref{wolfe2}, we employed the algorithm proposed in \cite{PerezPrudente2019}. This algorithm utilizes quadratic/cubic polynomial interpolations of the objective functions, combining backtracking and extrapolation strategies, and is capable of finding step sizes in a finite number of iterations. Interpolation techniques were also used to calculate step sizes satisfying only the Armijo-type condition. We set $\rho=10^{-4}$, $\sigma=0.1$, and initialized $B_j^0$ as the identity matrix for all $j=1,\ldots,m$. Convergence was reported when $|\theta(x^k)| \leq 5 \times \texttt{eps}^{1/2}$, where $\texttt{eps}=2^{-52}\approx 2.22 \times 10^{-16}$ represents the machine precision. When this criterion is met, we consider the problem successfully solved. The maximum number of allowed iterations was set to 2000. If this limit is reached, it means an unsuccessful termination. Our codes are freely available at \url{https://github.com/lfprudente/GlobalBFGS}.

The chosen set of test problems consists of both convex and nonconvex multiobjective problems commonly found in the literature and coincides with the one used in \cite{Leandro&Danilo}. Table~\ref{tab:problems} presents their main characteristics: The first column contains the problem name, while the ``$n$'' and ``$m$'' columns provide the number of variables and objectives, respectively. The column ``Conv.'' indicates whether the problem is convex or not. For each problem, the starting points were chosen within a box defined as $\{x\in\R^n \mid \ell\leq x \leq u\}$, where the lower and upper bounds, denoted by $\ell$ and $u \in\R^n$, are presented in the last two columns of Table~\ref{tab:problems}. It is important to note that the boxes specified in the table were used solely for defining starting points and were not employed as constraints during the algorithmic processes. For detailed information regarding the references and corresponding formulations of each problem, we refer the reader to  \cite{Leandro&Danilo}.

\begin{table}[H]
\centering
\begin{threeparttable}{\scriptsize
\fontsize{7.0}{8.0}\selectfont
\renewcommand{\arraystretch}{\myscaletable}
\begin{tabular}{|cccccc|} 
\hline
\rowcolor[gray]{.85} Problem & $n$  & $m$ & Conv. & $\ell$ & $u$ \\ 
\hline              
AP1&       2 &  3 & Y & $(-10,-10)$ & $(10,10)$   \\ \rowcolor[gray]{\lightgray}
AP2&      1 &  2 & Y  & $-100$ & $100$ \\
AP3&      2 &  2 & N  & $(-100,-100)$ & $(100,100)$ \\ \rowcolor[gray]{\lightgray}
AP4&      3 &  3 & Y & $(-10,-10,-10)$ & $(10,10,10)$  \\
BK1&   2 &  2 & Y  & $(-5,-5)$  & $(10,10)$ \\  \rowcolor[gray]{\lightgray}
DD1&  5 &  2 & N & $(-20,\ldots,-20)$ & $(20,\ldots,20)$ \\
DGO1&    1 &  2 & N & $-10$  & $13$  \\  \rowcolor[gray]{\lightgray}
DGO2 &  1 & 2  & Y  & $-9$  & $9$ \\
DTLZ1&  7 &  3 &  N & $(0,\ldots,0)$ & $(1,\ldots,1)$\\   \rowcolor[gray]{\lightgray}
DTLZ2&  7 & 3  & N & $(0,\ldots,0)$ & $(1,\ldots,1)$  \\
DTLZ3&  7 & 3  & N   & $(0,\ldots,0)$ & $(1,\ldots,1)$ \\  \rowcolor[gray]{\lightgray}
DTLZ4&  7 & 3  &  N  & $(0,\ldots,0)$ & $(1,\ldots,1)$ \\
FA1&   3 & 3  & N &  $(0.01,0.01,0.01)$  & $(1,1,1)$ \\  \rowcolor[gray]{\lightgray}
Far1&  2 &  2 & N  & $(-1,-1)$  &  $(1,1)$\\ 
FDS &  5 & 3 &Y & $(-2,\ldots,-2)$ &  $(2,\ldots,2)$ \\   \rowcolor[gray]{\lightgray}
FF1 &      2 &  2 & N  & $(-1,-1)$ &  $(1,1)$ \\ 
Hil1&     2 &  2 & N  & $(0,0)$  &   $(1,1)$ \\  \rowcolor[gray]{\lightgray}
IKK1&   2 &  3 & Y  & $(-50,-50)$  & $(50,50)$   \\ 
IM1&  2 &  2 & N & $(1,1)$  & $(4,2)$\\  \rowcolor[gray]{\lightgray}
JOS1&    2 &  2 & Y& $(-100,\ldots,-100)$ & $(100,\ldots,100)$  \\ 
JOS4& 20 & 2  &  N & $(-100,\ldots,-100)$ & $(100,\ldots,100)$ \\  \rowcolor[gray]{\lightgray}
KW2 &  2 &  2 & N  & $(-3,-3)$  &  $(3,3)$ \\   
LE1&    2 &  2 & N  & $(1,1)$  & $(10,10)$    \\   \rowcolor[gray]{\lightgray}
Lov1& 2 & 2 & Y & $(-10,-10)$ & $(10,10)$\\
Lov2&  2 & 2  & N  & $(-0.75,-0.75)$ & $(0.75,0.75)$ \\  \rowcolor[gray]{\lightgray}
Lov3&  2 & 2 & N & $(-20,-20)$ & $(20,20)$\\\ 
Lov4& 2 & 2 & N & $(-20,-20)$ & $(20,20)$\\   \rowcolor[gray]{\lightgray}
Lov5& 3 & 2 & N & $(-2,-2,-2)$ & $(2,2,2)$ \\
Lov6&  6 & 2  & N & $(0.1,-0.16,\ldots,-0.16)$ & $(0.425,0.16,\ldots,0.16)$ \\  \rowcolor[gray]{\lightgray}
LTDZ&   3 &  3 &  N & $(0,0,0)$  & $(1,1,1)$ \\ 
MGH9&  3 & 15  & N & $(-2,-2,-2)$ &  $(2,2,2)$  \\    \rowcolor[gray]{\lightgray}
MGH16&    4 &  5 & N  & $(-25,-5,-5,-1)$ & $(25,5,5,1)$ \\ 
MGH26 &   4 &  4 & N & $(-1,-1,-1-1)$ &  $(1,1,1,1)$ \\   \rowcolor[gray]{\lightgray}
MGH33 &    10 &  10 & Y & $(-1,\ldots,-1)$ &  $(1,\ldots,1)$ \\   
MHHM2&   2 &  3 & Y & $(0,0)$  & $(1,1)$ \\   \rowcolor[gray]{\lightgray}
MLF1&  1 & 2  &  N & $0$  & $20$ \\
MLF2&     2 &  2 & N   & $(-100,-100)$ & $(100,100)$\\   \rowcolor[gray]{\lightgray}
MMR1  &      2 &  2 & N  & $(0.1,0)$   & $(1,1)$   \\
MMR2&  2 & 2  &  N  & $(0,0)$   & $(1,1)$ \\   \rowcolor[gray]{\lightgray}
MMR3  &      2 &  2 & N & $(-\pi,-\pi)$  & $(\pi,\pi)$  \\ 
MMR4& 3 & 2  & N   & $(0,0,0)$   & $(4,4,4)$ \\ \rowcolor[gray]{\lightgray}
MOP2 &       2 &  2 & N  & $(-4,-4)$ &  $(4,4)$   \\  
MOP3 &     2 &  2 & N  & $(-\pi,-\pi)$  & $(\pi,\pi)$  \\ \rowcolor[gray]{\lightgray}
MOP5&      2 &  3 & N & $(-30,-30)$ &  $(30,30)$  \\
MOP6&  2 & 2  & N   & $(0,0)$ &  $(1,1)$ \\    \rowcolor[gray]{\lightgray}
MOP7&      2 &  3 & Y & $(-400,-400)$  & $(400,400)$  \\  
PNR&   2 &  2 & Y & $(-2,-2)$  &   $(2,2)$ \\    \rowcolor[gray]{\lightgray}
QV1&    10 &  2 & N & $(0.01,\ldots,0.01)$  & $(5,\ldots,5)$  \\  
SD &  4 & 2  &  Y & $(1,\sqrt{2},\sqrt{2},1)$  & $(3,3,3,3)$ \\   \rowcolor[gray]{\lightgray}
SK1 &       1 &  2 & N  & $-100$ &$100$  \\ 
SK2 &       4 &  2 & N & $(-10,-10,-10,-10)$ & $(10,10,10,10)$ \\    \rowcolor[gray]{\lightgray}
SLCDT1&     2 &  2 & N & $(-1.5,-1.5)$  &  $(1.5,1.5)$  \\ 
SLCDT2&  10 &  3 & Y  & $(-1,\ldots,-1)$  &  $(1,\ldots,1)$  \\    \rowcolor[gray]{\lightgray}
SP1&      2 &  2 & Y   &  $(-100,-100)$ & $(100,100)$   \\  
SSFYY2&      1 &  2 & N  &  $-100$ & $100$  \\    \rowcolor[gray]{\lightgray}
TKLY1&4 & 2  & N  &  $(0.1,0,0,0)$ & $(1,1,1,1)$  \\
Toi4&   4 &  2 & Y & $(-2,-2,-2,-2)$  & $(5,5,5,5)$    \\     \rowcolor[gray]{\lightgray}
Toi8&   3 &  3 & Y  & $(-1,-1,-1,-1)$  & $(1,1,1,1)$   \\
Toi9 & 4   & 4 & N  & $(-1,-1,-1,-1)$  & $(1,1,1,1)$  \\     \rowcolor[gray]{\lightgray}
Toi10 & 4 & 3 & N & $(-2,-2,-2,-2)$  & $(2,2,2,2)$        \\  
VU1&    2 &  2 & N & $(-3,-3)$ &$(3,3)$  \\     \rowcolor[gray]{\lightgray}
VU2& 2  &  2 & Y & $(-3,-3)$ &$(3,3)$  \\
ZDT1&30 & 2  & Y  & $(0,\ldots,0)$ & $(1,\ldots,1)$   \\   \rowcolor[gray]{\lightgray}
ZDT2&  30 & 2  & N  &  $(0.01,\ldots,0.01)$ & $(1,\ldots,1)$   \\ 
ZDT3& 30  & 2  & N   &  $(0.01,\ldots,0.01)$ & $(1,\ldots,1)$  \\   \rowcolor[gray]{\lightgray}
ZDT4&  30 &  2 & N  &  $(0.01,-5,\ldots,-5)$ & $(1,5,\ldots,5)$  \\ 
ZDT6& 10 & 2  & N  &  $(0.01,\ldots,0.01)$ & $(1,\ldots,1)$   \\   \rowcolor[gray]{\lightgray}
ZLT1&  10 &  5 & Y  & $(-1000,\ldots,-1000)$ &$(1000,\ldots,1000)$  \\  
\hline
\end{tabular}}
\end{threeparttable}
\caption{List of test problems.}
\label{tab:problems}
\end{table}

In multiobjective optimization, the primary objective is to estimate the Pareto frontier of a given problem. A commonly used strategy is to execute the algorithm from multiple distinct starting points and collect the Pareto optimal points found. Thus, each problem listed in Table~\ref{tab:problems} was addressed by running all algorithms from 300 randomly generated starting points within their respective boxes. In this first stage, each problem/starting point was considered an independent instance and a run was considered successful if an approximate critical point was found, regardless of the objective functions values.
Figure~\ref{fig:results} presents the comparison of the algorithms in terms of CPU time using a performance profile \cite{dolan2002benchmarking}.
As can be seen, Algorithm~\ref{alg:MBFGS} and the BFGS-Wolfe algorithm exhibited virtually identical performance, outperforming the Cautious BFGS-Armijo algorithm. All methods proved to be robust, successfully solving more than $98\%$ of the problem instances.
It is worth noting that although the BFGS-Wolfe algorithm enjoys (theoretical) global convergence only under convexity assumptions, it also performs exceptionally well for nonconvex problems, which is consistent with observations in the scalar case.

\noindent\begin{figure}[h!]
\centering
\fontsize{6.0}{3.0}\selectfont
\includegraphics[scale=\myscale]{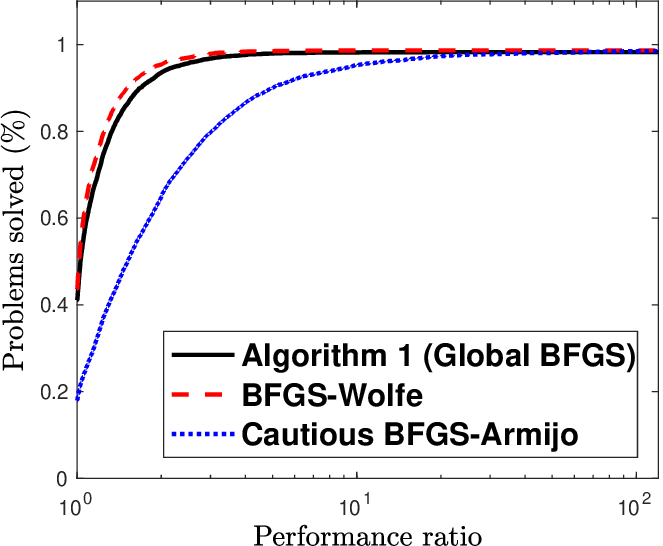}
\caption{Performance profiles considering 300 starting points for each test problem using the CPU time as performance measurement.}
\label{fig:results}
\end{figure}

In the following, we evaluate the algorithms based on their ability to properly generate Pareto frontiers.
To assess this, we employ the widely recognized {\it Purity} and ($\Gamma$ and $\Delta$) {\it Spread} metrics.
In summary, the Purity metric measures the solver's ability to identify points on the Pareto frontier, while the Spread metric evaluates the distribution quality of the obtained Pareto frontier.
For a detailed explanation of these metrics and their application together with performance profiles, we refer the reader to \cite{doi:10.1137/10079731X}.
It is important to note that, at this stage, data referring to all starting points are combined for each problem, taking into account the objective function values found.
The results in Figure~\ref{fig:metrics} indicate that Algorithm~\ref{alg:MBFGS} performed slightly better in terms of the Purity and $\Delta$-Spread metrics, with no significant difference observed for the $\Gamma$-Spread metric among the three algorithms.

\noindent\begin{figure}[h!]
\centering
 \begin{tabular}{ccc}
 (a) Purity &(b) $\Gamma$-Spread &(c) $\Delta$-Spread \\
\includegraphics[scale=\myscale]{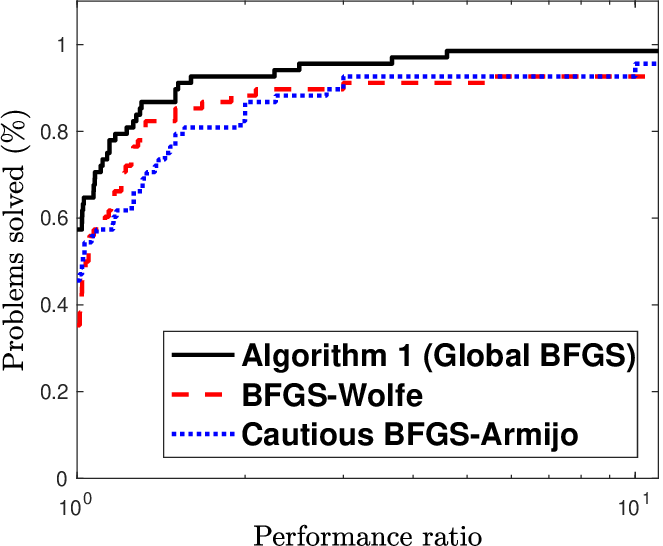}& \includegraphics[scale=\myscale]{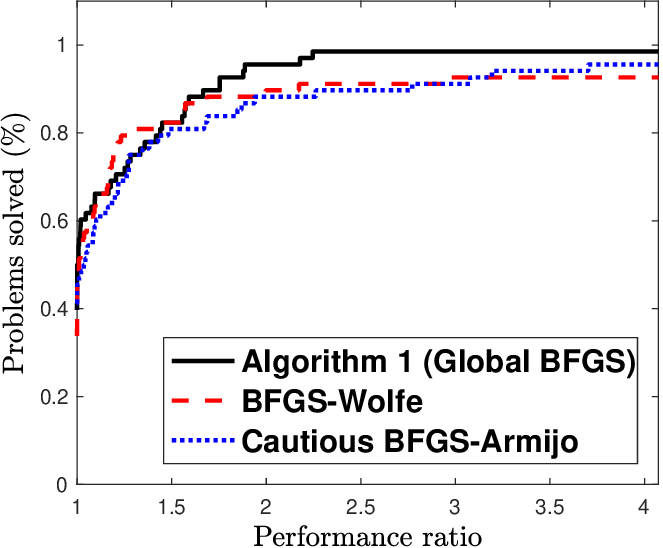}&\includegraphics[scale=\myscale]{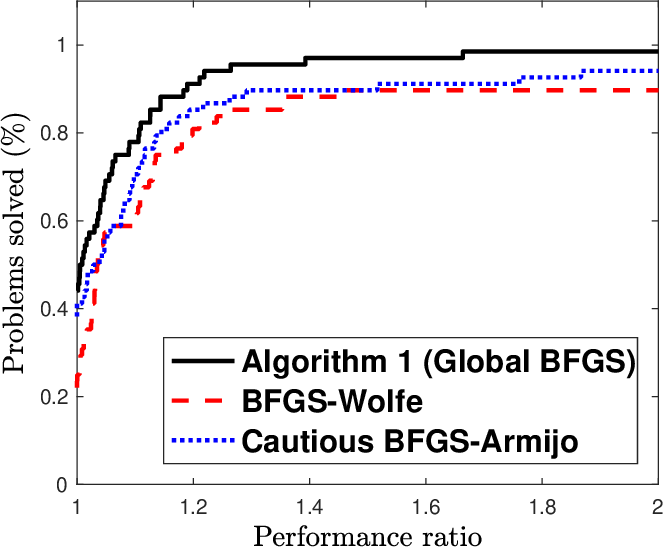}\\
\end{tabular}
\caption{Metric performance profiles: (a) Purity; (b) $\Gamma$-Spread; (c) $\Delta$-Spread.}
\label{fig:metrics}
\end{figure}

The numerical results allow us to conclude that the modifications made to the BFGS method to ensure global convergence for nonconvex problems do not compromise its practical performance.

\section{Final remarks} \label{sec:final}

 Based on the work of Li and Fukushima~\cite{Li-Fu-MBFGS}, we presented a modified BFGS scheme that achieves global convergence without relying on convexity assumptions for the objective function $F$. The global convergence analysis depends only on the requirement of $F$ having continuous Lipschitz gradients. Furthermore, we showed that by appropriately selecting $r_j^k$ to satisfy Assumption~\ref{assumption3} and under suitable conditions, the rate of convergence becomes superlinear. We also discussed some practical choices for $r_j^k$. The introduced modifications preserve the simplicity and practical efficiency of the BFGS method. It is worth emphasizing that the assumptions considered in our approach are natural extensions of those commonly employed in the context of scalar-valued optimization.

 \vspace{12pt}
\noindent{\bf Data availability statement}\\
The codes supporting the numerical experiments are freely available in the Github repository, \url{https://github.com/lfprudente/GlobalBFGS}.

 \vspace{12pt}
\noindent{\bf Conflicts of interest}\\
The authors declare that they have no conflict of interest.

\appendix
\section{Appendix}

In the main body of the article, we have chosen to exclude proofs that can be readily derived from existing sources in order to enhance the overall readability of the text. However, in this appendix, we provide these proofs to ensure self-contained completeness.

\vspace{5pt}
\noindent{\bf Notation.} The cardinality of a set $C$ is denoted by $|C|$.
The ceiling and floor functions are denoted by $\lceil \cdot \rceil$ and $\lfloor \cdot \rfloor$, respectively; i.e., if $x\in\R$, then $\lceil x \rceil$ is the least integer greater than or equal to $x$ and $\lfloor x \rfloor$ is the greatest integer less than or equal to $x$.
The notation $\varphi(t):=o(t)$ for $t>0$ means that $\lim_{t\to 0}\varphi(t)/t=0$.

\subsection{Proofs of Section~\ref{sec:algorithm}}

Throughout this section, we assume that Assumption~\ref{assumption1} holds.

\begin{proofof}{Proposition~\ref{prop:zout}}
It follows from \eqref{wolfe2}, the definition of $\D(\cdot,\cdot)$, the Cauchy-Schwarz inequality, and Assumption~\ref{assumption1}{\it (iii)} that
  \begin{align*}
  -(1-\sigma)\D(x^k,d^k) &\leq  \D(x^{k+1},d^k) - \D(x^k,d^k) \\
  & \leq \max_{j = 1,\dots,m }\left(\nabla F_j(x^{k+1})-\nabla F_j(x^k)\right)^\top d^k \\
  &\leq \max_{j = 1,\dots,m }\|\nabla F_j(x^{k+1})-\nabla F_j(x^k)\| \|d^k\|\leq L\alpha_k \|d^k\|^2,
    \end{align*}
    where the second inequality follows from the fact that, for any $u,v\in\R^m$, we have $\max_j(u_j - v_j) \geq \max_ju_j - \max_jv_j$.
  Hence,
  \begin{equation}\label{spfc}
    \dfrac{\D(x^k,d^k)^2}{\|d^k\|^2} \leq  -\dfrac{L}{(1-\sigma)}\alpha_k\D(x^k,d^k).
  \end{equation}
  Now, since $\{x^k\}\subset\Ls$, Assumption~\ref{assumption1}{\it (i)--(ii)} implies the existence of ${\cal F}\in\R$ such that $F_j(x^k)\geq {\cal F}$ for all $k\geq 0$ and $j=1,\ldots,m$. Therefore, by \eqref{wolfe1}, we have
  $${\cal F}\leq F_j(x^{k+1}) \leq F_j(x^0)+\rho \sum_{\ell=0}^k \alpha_\ell \D(x^\ell,d^\ell),  \quad \forall j = 1,\ldots,m.$$
Some algebraic manipulations yields
  $$-\dfrac{L}{\rho(1-\sigma)} \min_{j = 1,\dots,m}\left\{{\cal F}-F_j(x^0)\right\} \geq -\dfrac{L}{(1-\sigma)}  \sum_{\ell=0}^k \alpha_\ell \D(x^\ell,d^\ell)>0.$$
  Therefore,
  $$-\dfrac{L}{(1-\sigma)}  \sum_{k\geq 0} \alpha_k \D(x^k,d^k)<\infty,$$
  which together with \eqref{spfc} gives \eqref{Zout}.
\end{proofof}

To prove Proposition~\ref{l:poa}, we will make use of function \eqref{trdet}. Let us define
$$q_j^k:=\dfrac{(s^k)^\top B_j^ks^k}{(s^k)^\top s^k},  \quad \forall j = 1,\ldots,m.$$
Thus, from the same arguments that led to \eqref{fortaleza}, we obtain
\begin{equation}\label{santos}
 \psi(B_j^{k+1})= \psi(B_j^{k}) +\left[\dfrac{\|\gamma_j^k\|^2 }{(\gamma_j^{k})^\top s^k} - \ln\left(\dfrac{(\gamma_j^{k})^\top s^k}{(s^k)^\top s^k}\right) -1\right] -\xi_j^k,
\end{equation}
where
 $$\xi_j^k:=-\ln(\cos^2\beta_j^k)-\left[1 - \dfrac{q_j^k}{\cos^2\beta_j^k} + \ln\bigg(\dfrac{q_j^k}{\cos^2\beta_j^k}\bigg)\right].$$
 Note from Lemma~\ref{l:natal}{\it (i)} that $\xi_j^k\geq 0$.

\begin{proofof}{Proposition~\ref{l:poa}}
Let $k\geq 1$ and $p \in (0,1)$ be given and set $\varepsilon\coloneqq 1-p$ and $\bar{p}\coloneqq 1-\varepsilon/m$. Let $j\in\{1,\ldots, m\}$ be an arbitrary index. From \eqref{santos} and \eqref{atletico}, we have
 $$ \psi(B_j^{k+1})\leq \psi(B_j^{0}) +\left[C_2 - \ln(C_1) -1\right](k+1) -\sum_{\ell=0}^k\xi_j^\ell.$$
Therefore, since $\psi(B_j^{k+1})>0$, we obtain
  $$ \dfrac{1}{k+1}\sum_{\ell=0}^k\xi_j^\ell\leq \dfrac{\psi(B_j^{0})}{k+1} +\left[C_2 - \ln(C_1) -1\right].$$
Let ${\cal J}_j^{k}$ be the set consisting of the $\lceil \bar{p}(k+1) \rceil$ indices corresponding to the $\lceil \bar{p}(k+1) \rceil$ smallest values of $\xi_j^\ell$, for $\ell\leq k$, and define $\bar{\xi}_j^{k}:=\max_{\ell\in{\cal J}_j^{k}}\xi_j^\ell$. Then,
  $$ \dfrac{1}{k+1}\sum_{\ell=0}^k\xi_j^\ell\geq  \dfrac{1}{k+1}\left[\bar{\xi}_j^{k}+\sum_{\ell=0, \ell\notin{\cal J}_j^{k}}^k \xi_j^\ell\right]\geq \dfrac{1}{k+1} \left[\bar{\xi}_j^{k}+\bar{\xi}_j^{k}(k+1-\lceil \bar{p}(k+1) \rceil)\right]\geq \bar{\xi}_j^{k}(1-\bar{p}),$$
  where the last inequality is due to $\lceil \bar{p}(k+1) \rceil\leq \bar{p}(k+1)+1.$
By combining the above two inequalities, we get, for all $\ell\in{\cal J}_j^{k}$, 
$$\xi_j^\ell\leq \bar{\xi}_j^{k}\leq \dfrac{1}{1-\bar{p}}\left[\psi(B_j^{0})+C_2 - \ln(C_1) -1\right]=:\zeta_j.$$
Therefore, by the definition of $\xi_j^\ell$, we obtain, for all $\ell\in{\cal J}_j^{k}$, 
$$-\ln(\cos^2\beta_j^\ell)\leq\xi_j^\ell\leq \zeta_j,$$
and hence
$$\cos\beta_j^\ell\geq e^{-\zeta_j/2}=:\delta_j.$$
This means that $\cos\beta_j^{\ell} \geq \delta_j$ for at least $\lceil \bar{p}(k+1) \rceil$ values of $\ell \in \{0,1,\ldots,k\}$.

Now, let us define $\delta\coloneqq\min_{j=1,\ldots,m}\delta_j$ and, for all $j=1,\ldots,m$,
$${\cal G}_j^k\coloneqq\{\ell\in\{0,1,\ldots,k\}\mid \cos\beta_j^{\ell} \geq \delta\} \quad \mbox{and} \quad {\cal B}_j^k\coloneqq\{\ell\in\{0,1,\ldots,k\}\mid \cos\beta_j^{\ell} < \delta\}.$$
It is easy to see that ${\cal J}_j^{k}\subset {\cal G}_j^k$, ${\cal G}_j^k\cap{\cal B}_j^k=\emptyset$ and $|{\cal G}_j^k|+|{\cal B}_j^k|=k+1$. Therefore, by the definition of $\bar{p}$ and using some properties of the ceiling and floor functions, we have, for all $j=1,\ldots,m$,
	\[
	|{\cal G}_j^k| \geq |{\cal J}_j^{k}|= \lceil \bar{p}(k+1) \rceil = (k+1) + \lceil - \frac{\varepsilon}{m}(k+1)\rceil = (k+1) - \lfloor \frac{\varepsilon}{m}(k+1)\rfloor,
	\]
	and hence $|{\cal B}_j^k|\leq \lfloor \frac{\varepsilon}{m}(k+1)\rfloor$.
	Thus, 
	$$|\mathop \cup_{j=1}^m {\cal B}_j^k|\leq  m \lfloor \frac{\varepsilon}{m}(k+1)\rfloor\leq \varepsilon (k+1).$$
	As consequence, since we also have $|\mathop \cap_{j=1}^m {\cal G}_j^k|\ +  |\mathop \cup_{j=1}^m {\cal B}_j^k|=k+1$, by using the definition of $\varepsilon$, it follows that
	\[
	|\mathop \cap_{j=1}^m {\cal G}_j^k| \geq (k+1) - \varepsilon(k+1) = (1-\varepsilon)(k+1)=p(k+1),
	\]
	which concludes the proof.
\end{proofof}

\subsection{Proofs of Section~\ref{sec:superlinear}}

In this section, we make use of Assumption~\ref{assumption2}. In particular, and without loss of generality, we assume that $\{x^k\}\subset U$, where $U$ is a neighborhood of $\xs$ such that \eqref{lip}  and \eqref{sconvex} hold.

\subsubsection{Proof of Proposition~\ref{prop:linear}}

We start with some auxiliary technical results.

\begin{lemma}\label{prop:pseud:angle}
	Suppose that Assumption~\ref{assumption2} holds. 
Let $\beta_j^k$ be the angle between the vectors $s^k$ and $B_j^ks^k$, for all $k\geq 0$ and $j=1,\ldots,m$.
Then, for all $k\geq 0$,
	\begin{equation*}\label{pseud-angle-condit}
		\D(x^k,d^k) \leq -\dfrac{\delta_k}{2}\|d^k\| \| \dsd(x^k)\|, 
	\end{equation*}
	where $\delta_k \coloneqq \min_{j =1,\dots,m } \cos\beta_j^k$.
\end{lemma}
\begin{proof}
For a given $k\geq 0$, by using the definitions of $\delta_k$, $\cos\beta_j^k$, and $s^k$, we obtain
\[
\delta_k\leq  \cos\beta_j^k = \frac{(s^k)^\top B_j^ks^k}{\|s^k\|\|B_j^ks^k\|}=\frac{(d^k)^\top B_j^kd^k}{\|d^k\|\|B_j^kd^k\|}, \quad  \forall j=1,\ldots,m.
\]
Therefore, from Lemma~\ref{l:pareto}{\it (ii)} and \eqref{thetaN}, we have
\[
-\D(x^k,d^k)>-\theta(x^k) =  \frac{1}{2} \sum_{j=1}^{m} \lambda_j^k (d^k)^\top B_j^k d^k \geq \frac{\delta_k}{2}  \|d^k\| \sum_{j=1}^{m} \lambda_j^k \|B_j^kd^k\|. 
\]
Applying the triangle inequality, together with \eqref{dN}, \eqref{lambdaN}, and Lemma~\ref{l:vtheta}{\it (iv}), we obtain:
\[-\D(x^k,d^k) \geq \frac{ \delta_k}{2} \|d^k\| \|\sum_{j=1}^{m} \lambda_j^k B_j^kd^k\| =  \frac{\delta_k}{2}  \|d^k\| \|\sum_{j=1}^{m} \lambda_j^k \nabla F_j(x^k)\|\geq  \frac{ \delta_k}{2} \|d^k\| \|\dsd(x^k)\|.\]
\end{proof}

\begin{lemma}\label{peru}
	Suppose that Assumption~\ref{assumption2} holds. Then, for all  $k\geq 0$, we have:
	\begin{itemize}
		\item[(i)] $\|x^k-x^{\ast}\|\leq \dfrac{2}{\underline{L}}\|\dsd(x^k)\|$;
		\item[(ii)] $\|s^k\| \geq \dfrac{(1-\sigma)}{2L}\delta_k\|\dsd(x^k)\|$, where $\delta_k$ is given as in Lemma~\ref{prop:pseud:angle};
		\item[(iii)] $\dfrac{(\gamma_j^{k})^\top s^k}{\|s^k\|^2}\geq \underline{L}$,  for all $j=1,\ldots,m$;
		\item[(iv)] $\dfrac{\|\gamma_j^{k}\|^2}{(\gamma_j^{k})^\top s^k}\leq \dfrac{(2L+ \bar{\vartheta}\bar{c})^2}{\underline{L}}$, for all $j=1,\ldots,m$ and some constant $\bar{c}>0.$
	\end{itemize}
\end{lemma}

\begin{proof}
Consider part {\it (i)}. 
For a given value of $k\geq 0$, consider $\lsd(x^k)\in\R^m$ as in \eqref{lsd}--\eqref{dsd}, and define the scalar-valued function  $F_{SD}\colon\R^n\to\R$ as follows: 
	$$F_{SD}(x)\coloneqq \sum_{j=1}^m \lambda^{SD}_j(x^k) F_j(x).$$
	Therefore, by taking $z\coloneqq x^{\ast}-x^k$, it follows from \eqref{lsd} and \eqref{sconvex} that
	\[\int_{0}^{1} (1 - \tau)z^\top  \nabla^2 F_{SD}(x^k + \tau z)zd\tau \geq \dfrac{\underline{L}}{2}\|z\|^2.\]
	Evaluating this integral (which can be done by integration by parts), and considering that $\dsd(x^k)=-\nabla F_{SD}(x^k)$, we obtain
	\[F_{SD}(x^{\ast})-F_{SD}(x^k)+\dsd(x^k)^\top (x^{\ast}-x^k)\geq \dfrac{\underline{L}}{2}\|x^{\ast}-x^k\|^2.\]
	Given that $F_j(x^{\ast})\leq F_j(x^k)$ for all $j=1,\ldots,m$, we have $F_{SD}(x^{\ast})-F_{SD}(x^k)\leq 0$ and thus
	\[\dfrac{\underline{L}}{2}\|x^{\ast}-x^k\|^2\leq \dsd(x^k)^\top (x^{\ast}-x^k)\leq \|\dsd(x^k)\| \|x^{\ast}-x^k\|,\]
	which proves  part {\it (i)}.
	
	Consider part {\it (ii)}. 
	By using \eqref{wolfe2} and the definitions of ${\cal D}(\cdot,\cdot)$ and $y_j^k$, we obtain
$$-(1-\sigma)\D(x^k,d^k)  \leq \D(x^{k+1},d^k) - \D(x^k,d^k)\leq  \max_{j = 1,\dots,m }(y_j^k)^\top d^k,$$
	which, together with \eqref{uruguai}, yields
	\[
	-(1-\sigma)\D(x^k,d^k)  \leq  \max_{j = 1,\dots,m }(s^k)^\top  \bar{G}_j^k d^k =\alpha_k  \max_{j = 1,\dots,m }(d^k)^\top  \bar{G}_j^k d^k \leq L \alpha_k \|d^k\|^2=L\|s^k\|\|d^k\|,
	\]
	where  the latter inequality comes from \eqref{sconvex}. Therefore, taking into account that $\sigma<1$,  by Lemma~\ref{prop:pseud:angle}, we obtain
	\[
	(1-\sigma)\dfrac{\delta_k}{2}\|d^k\| \| \dsd(x^k)\| \leq L\|s^k\|\|d^k\|,
	\]
	which gives the desired inequality.
	
	Part {\it (iii)} is a direct consequence of \eqref{joaopessoa}. Finally, consider part {\it (iv)}. From  \eqref{etajk} and \eqref{sconvex}, we have
$$|\eta_j^k|\leq \dfrac{\|y_{j}^k\|}{\|s^k\|} = \dfrac{\|\nabla F_{j}(x^{k+1})-\nabla F_j(x^k)\|}{\|x^{k+1}-x^{k}\|}\leq L.$$
Furthermore, since $\{x^k\}\subset U$, by \eqref{rjk} and using continuity arguments, there exists a constant $\bar{c}>0$ such that
$$0\leq r_j^k\leq |\eta_j^k| + \vartheta_k \|\sum_{i=1}^m\mu_i^k \nabla F_i(x^k)\| \leq L+ \bar{\vartheta}\bar{c},$$
and hence, by \eqref{gammajk},
$$\|\gamma_j^k\| \leq  \|y_j^k\| + r_j^k \|s^k\| = \left(\dfrac{\|y_j^k\|}{\|s^k\|} + r_j^k \right) \|s^k\| \leq (2L+ \bar{\vartheta}\bar{c}) \|s^k\|.$$
Therefore, using the inequality in part {\it (iii)}, we obtain
$$\dfrac{\|\gamma_j^{k}\|^2}{(\gamma_j^{k})^\top s^k}=\dfrac{\|\gamma_j^{k}\|^2}{\|s^k\|^2} \dfrac{\|s^k\|^2}{(\gamma_j^{k})^\top s^k}\leq \dfrac{(2L+ \bar{\vartheta}\bar{c})^2}{\underline{L}}, \quad \forall j=1,\ldots,m,$$
concluding the proof.
\end{proof}

We are now able to prove Proposition~\ref{prop:linear}

\begin{proofof}{Proposition~\ref{prop:linear}}
	Let $\lambda^{SD}(\xs)\in\R^m$ be a steepest descent multiplier associated with $\xs$ as in \eqref{lsd}--\eqref{dsd}, and define the scalar-valued function $F_{\ast}\colon\R^n\to\R$ as follows:
	\[F_{\ast}(x)\coloneqq \sum_{j=1}^m \lambda^{SD}_j(\xs) F_j(x).\] 
	Note that 
	\begin{equation}\label{equador}
		\nabla F_{\ast}(\xs)= \sum_{j=1}^m \lambda^{SD}_j(\xs) \nabla F_j(\xs) = -\dsd(\xs) = 0,
	\end{equation}
	where the last equality comes from Lemma~\ref{l:vtheta}{\it (i)}.
	Now, by using \eqref{sconvex}, we obtain
$$\nabla F_j(\xs)^\top (x^k - \xs) + \dfrac{\underline{L}}{2}\|x^k - \xs\|^2  \leq F_j(x^k) - F_j(\xs) \leq \nabla F_j(\xs)^\top (x^k - \xs) + \dfrac{L}{2}\|x^k - \xs\|^2,$$
	for all $j=1,\ldots,m$ and for all $k\geq 0$. By multiplying this expression by $\lambda_j^{SD}(\xs)$, summing over all indices $j=1,\ldots,m$, and taking into account \eqref{lsd} and \eqref{equador}, we obtain
	\begin{equation}\label{chile}
		\dfrac{\underline{L}}{2}\|x^k - \xs\|^2  \leq F_{\ast}(x^k) - F_{\ast}(\xs) \leq   \dfrac{L}{2}\|x^k - \xs\|^2, \quad \forall k\geq 0.
	\end{equation}
	From the right hand side of \eqref{chile} and Lemma~\ref{peru}{\it (i)}, we obtain
	\begin{equation}\label{colombia}
		F_{\ast}(x^k) - F_{\ast}(\xs) \leq   \dfrac{2L}{\underline{L}^2}\|\dsd(x^k)\|^2, \quad \forall k\geq 0.
	\end{equation}
	On the other hand, \eqref{wolfe1} gives
	\[F_{\ast}(x^{k+1})  \leq F_{\ast}(x^k)+\rho  \alpha_k \D(x^k,d^k), \quad \forall k\geq 0.\]
	Therefore, from Lemma~\ref{prop:pseud:angle} and Lemma~\ref{peru}{\it(ii)}, we have
	\[
	F_{\ast}(x^{k+1})  \leq F_{\ast}(x^k) -  \dfrac{\rho }{2}\delta_k\|s^k\| \| \dsd(x^k)\|\leq F_{\ast}(x^k) -  \dfrac{\rho (1-\sigma)}{4L}\delta_k^2 \| \dsd(x^k)\|^2,  \quad \forall k\geq 0.
	\]
Hence, by subtracting the term $F_{\ast}(\xs)$ in both sides of the latter inequality, and using \eqref{colombia}, we obtain
	\begin{equation}\label{venezuela}
		F_{\ast}(x^{k+1}) - F_{\ast}(\xs) \leq  \left( 1 -  \dfrac{\rho (1-\sigma)\underline{L}^2}{8L^2}\delta_k^2  \right) \left(F_{\ast}(x^k)-F_{\ast}(\xs)\right), \quad \forall k\geq 0.
	\end{equation}
	For each $k\geq 0$, define $\bar r_k\coloneqq 1-\rho (1-\sigma)\underline{L}^2\delta_k^2/(8L^2)$. It is easy to see that $\bar r_k\in (0,1]$, for all $k\geq 0$.

	Now, given $p\in(0,1)$, we can invoke Lemma~\ref{peru}{\it (iii)}--{\it (iv)} to apply Proposition \ref{l:poa}. This implies that there exists a constant $\delta > 0$ such that, for any $k\geq 1$, the number of elements $\ell\in\{0,1,\ldots,k\}$ for which $\delta_{\ell} \geq \delta$ is at least $\lceil p(k+1) \rceil$. By defining ${\cal G}_k\coloneqq\{\ell\in\{0,1,\ldots,k\}\mid \delta_{\ell} \geq \delta\}$, we have $|{\cal G}_k|\geq \lceil p(k+1) \rceil$ and
	\[
	\bar r_{\ell}\leq 1 -  \dfrac{\rho (1-\sigma)\underline{L}^2\delta^2}{8L^2}    \coloneqq \bar r<1, \quad \forall\ell\in {\cal G}_k.
	\] 
Thus, from \eqref{venezuela} and considering that  $F_{\ast}(x^0) - F_{\ast}(\xs) >0$, we obtain, for all $k\geq 1$,
	\begin{align*}
		F_{\ast}(x^{k+1}) - F_{\ast}(\xs)& \leq \left[\prod_{\ell = 0}^{k}  {\bar r}_{\ell}\right] \left(F_{\ast}(x^0) - F_{\ast}(\xs) \right)  \leq \left[\prod_{\ell \in {\cal G}_k}{\bar r}_{\ell} \right]  \left(F_{\ast}(x^0) - F_{\ast}(\xs) \right) \\
		&  \leq \left[\prod_{\ell \in {\cal G}_k}{\bar r} \right]  \left(F_{\ast}(x^0) - F_{\ast}(\xs) \right) \leq {\bar r}^{ \lceil p(k+1) \rceil} \left(F_{\ast}(x^0) - F_{\ast}(\xs) \right),
	\end{align*}
where the second inequality follows from the fact that ${\bar r}_{\ell}\leq 1$ for all $\ell\notin {\cal G}_k$. 
Therefore, by taking $r:={\bar r}^p$, we obtain
\[
		F_{\ast}(x^{k+1}) - F_{\ast}(\xs)\leq r^{k+1}\left(F_{\ast}(x^0) - F_{\ast}(\xs) \right), \quad \forall k\geq 1. 
\]
Combining this with the left hand side of \eqref{chile}, we find
	\[
	\|x^{k+1} - \xs\|^\nu   \leq  \left[\dfrac{2}{\underline{L}}\left(F_{\ast}(x^0) - F_{\ast}(\xs) \right)\right]^{\nu/2}  (r^{\nu/2})^{k+1}.
	\]
Finally, by summing this expression and taking into account that $r<1$, we conclude that  \eqref{argentina} holds.
\end{proofof}

\subsubsection{Proof of Theorem~\ref{teo:convergence}}

We start by introducing an auxiliary result.

\begin{lemma}\label{l:mexico}
	Suppose that Assumptions~\ref{assumption2} and \ref{assumption3} hold.  Then, there exists $\bar{a}>0$ such that
	\begin{equation}\label{inglaterra}
		  |\theta(x^k)|\geq \bar{a}\|d^k\|^2,
	\end{equation} 
	for all $k$ sufficiently large. Moreover, 
	\begin{equation}\label{limd}
		\lim_{k\to\infty}\|d^k\|=0.
	\end{equation} 
\end{lemma}
\begin{proof}
By choosing $\gamma\in(0,1)$ and recalling that $s^k=\alpha_k d^k$, it follows from \eqref{limq} that
	\begin{equation*}\label{espanha}
	  \dfrac{(d^k)^\top  B_j^{k}d^k}{(d^k)^\top   \HF_j(\xs) d^k} \geq 1-\gamma, \quad \forall j=1,\ldots,m,
	\end{equation*}
	for all $k$ sufficiently large. 
	Thus, by \eqref{sconvex}, we obtain
	\[ (d^k)^\top B_j^kd^k \geq \underline{L}(1- \gamma)\|d^k\|^2, \quad \forall j=1,\ldots,m,\]
	for all $k$ sufficiently large. Therefore, using \eqref{lambdaN} and \eqref{thetaN}, we have
	\[ |\theta(x^k)| = \frac{1}{2}  \sum_{j=1}^{m} \lambda_j^k (d^k)^\top  B_j^kd^k \geq \dfrac{\underline{L}(1- \gamma)}{2}\|d^k\|^2 ,\]
	for all $k$ sufficiently large. Defining $\bar{a}\coloneqq \underline{L}(1- \gamma)/2$, we establish \eqref{inglaterra}.
	Finally, by combining \eqref{inglaterra}, Lemma~\ref{l:pareto}{\it (ii)}, and Proposition~\ref{prop:zout}, we obtain
	\[ 0\leq \lim_{k\to \infty}\bar{a}\|d^k\| \leq \lim_{k\to \infty} \dfrac{|\theta(x^k)|}{\|d^k\|}\leq \lim_{k\to \infty} \dfrac{|\D(x^k,d^k)|}{\|d^k\|}=0, \]
	which concludes the proof.
\end{proof}

Recalling that $\lambda^k\in\R^m$ is the Lagrange multiplier associated to $x^k$ of problem \eqref{newtonproblem}  fulfilling \eqref{dN}--\eqref{lambdaN}, let us define
\begin{equation}\label{alemanha}
	F_{\lambda}^k(x) \coloneqq \sum_{j = 1}^{m} \lambda_j^k F_j(x) \quad \mbox{and} \quad B_{\lambda}^k \coloneqq \sum_{j = 1}^{m}\lambda_j^kB_j^{k}, \quad \forall k\geq 0.
\end{equation}
Next, we show that the sequence of functions $\{F_{\lambda}^k(x)\}_{k\geq0}$ fulfills a Dennis--Mor\'e-type condition.

\begin{theorem}\label{teo:DM}
	Suppose that Assumptions~\ref{assumption2} and \ref{assumption3} hold.
	For each $k\geq 0$, consider $F_{\lambda}^k\colon\R^n\to\R$ and $B_{\lambda}^k$ as in \eqref{alemanha}.Then,
	\begin{equation}\label{DM2}
		\lim_{k \to \infty} \dfrac{\| (B_{\lambda}^k - \HF_{\lambda}^k(\xs))d^k\|}{\|d^k\|}=0
	\end{equation}
	or, equivalently,
	\begin{equation}\label{DM3}
		\lim_{k \to \infty} \dfrac{\| \nabla F_{\lambda}^k(x^k) + \HF_{\lambda}^k(x^k)d^k\|}{\|d^k\|}=0.
	\end{equation}
\end{theorem}
\begin{proof}
	By \eqref{alemanha} and taking into account \eqref{lambdaN}, we have	
	\begin{align*}
		\lim_{k \to \infty} \dfrac{\| (B_{\lambda}^k - \nabla^2 F_{\lambda}^k(\xs))d^k\|}{\|d^k\|} 
		& \leq  \lim_{k \to \infty} \sum_{j=1}^m \lambda_j^k\dfrac{  \|(B_j^k - \HF_j(\xs))d^k\|}{\|d^k\|} \\
		& \leq  \lim_{k \to \infty} \max_{j=1,\ldots,m} \dfrac{  \|(B_j^k - \HF_j(\xs))d^k\|}{\|d^k\|},
	\end{align*}
	which, combined with \eqref{aracaju}, yields \eqref{DM2}. We proceed to show that \eqref{DM2} implies \eqref{DM3}.
	Firstly, considering \eqref{dN}, since $B_{\lambda}^kd^k=-\nabla F_{\lambda}^k(x^k) $, it follows that \eqref{DM3} is equivalent to
	\begin{equation}\label{DM}
		\lim_{k \to \infty} \dfrac{\| (B_{\lambda}^k - \HF_{\lambda}^k(x^k))d^k\|}{\|d^k\|}=0.
	\end{equation}
	Note that
	\begin{align*}\label{reptheca}
			\lim_{k \to \infty} \dfrac{\| (B_{\lambda}^k - \HF_{\lambda}^k(x^k))d^k\|}{\|d^k\|}  \leq &\lim_{k \to \infty}   \dfrac{\| (B_{\lambda}^k - \HF_{\lambda}^k(\xs))d^k\|}{\|d^k\|}+ \lim_{k \to \infty}   \| \HF_{\lambda}^k(\xs) - \HF_{\lambda}^k(x^k)\| 
	\end{align*}
	and, by using continuity arguments,
	\begin{align*}
		\lim_{k \to \infty}  \| \HF_{\lambda}^k(\xs) - \HF_{\lambda}^k(x^k)\| 
		&\leq   \lim_{k \to \infty} \sum_{j=1}^m \lambda_j^k \| \HF_j(\xs)- \HF_j(x^k)\|\\
		&\leq   \lim_{k \to \infty} \max_{j=1,\ldots,m}  \|\HF_j(\xs)- \HF_j(x^k)\|=0.
	\end{align*}
	Therefore, combining the two latter inequalities, we obtain \eqref{DM}.
	The proof that \eqref{DM3} implies \eqref{DM2} can be obtained similarly.
\end{proof}

The following result shows that the unit step size eventually satisfies the Wolfe conditions \eqref{wolfe1}--\eqref{wolfe2}.

\begin{theorem}\label{teo:unit}
Suppose that Assumptions~\ref{assumption2} and \ref{assumption3} hold. Then, the step size $\alpha_k =1$ is admissible for all $k$ sufficiently large.
\end{theorem}
\begin{proof}
	Let $j\in\{1,\ldots,m\}$ be an arbitrary index. It is easy to see that \eqref{aracaju} is equivalent to
	$$ \lim_{k \to\infty}\dfrac{\| (B_j^k - \nabla ^2F_j(x^k))d^k\|}{\|d^k\|} = 0.$$
	Thus, by Taylor’s theorem, it follows that
	\begin{align*}
		F_j(x^k + d^k)  = & F_j(x^k) + \nabla F_j(x^k)^\top d^k + \dfrac{1}{2}(d^k)^\top B_j^kd^k + \dfrac{1}{2}(d^k)^\top \left(\HF_j(x^k) - B_j^k\right) d^k + o(\|d^k\|^2)\\
		= & F_j(x^k) + \nabla F_j(x^k)^\top d^k + \dfrac{1}{2}(d^k)^\top B_j^kd^k + o(\|d^k\|^2),
	\end{align*}
	Therefore, by using \eqref{theta} and setting $t\coloneqq 2\rho<1$, we have
	\[F_j(x^k + d^k) \le F_j(x^k) + t\theta(x^k) +(1-t)\theta(x^k) + o(\|d^k\|^2).\]
	Consequently, according to \eqref{inglaterra}, for sufficiently large $k$,
$$F_j(x^k + d^k) \leq  F_j(x^k) + t\theta(x^k) +\left[-\bar{a}(1-t)+ \dfrac{o(\|d^k\|^2)}{\|d^k\|^2}\right]\|d^k\|^2.$$
As the term in square brackets is negative for $k$ large enough, we conclude that 
	\[
	F_j(x^k + d^k) \leq F_j(x^k) + t\theta(x^k).
	\]
	On the other hand, combining \eqref{dN}--\eqref{thetaN}, we find
	\begin{equation}\label{thetaalt}
	\theta(x^k)=\dfrac{1}{2}\sum_{j=1}^m\lambda_j^k\nabla F_j(x^k)^\top d^k \leq\dfrac{1}{2}\D(x^k,d^k).
	\end{equation}
	Hence, from the last two inequalities and the definition of $t$, we obtain
	\[ F_j(x^k + d^k) \leq F_j(x^k) + \rho\D(x^k,d^k),\]
	for all $k$ sufficiently large.  Given the arbitrary choice of $j\in\{1,\ldots,m\}$, we conclude that the step size $\alpha_k =1$ satisfies \eqref{wolfe1} for all sufficiently large $k$.
	
	Consider the curvature condition \eqref{wolfe2}. From the definition of $F_{\lambda}^k$ in \eqref{alemanha}, we have
	\begin{align*}
		-\sum_{j = 1}^{m}\lambda_j^k\nabla F_j(x^k)^\top d^k 
		& =  \sum_{j = 1}^{m}\lambda_j^k(d^k)^\top\nabla^2F_j(x^k)d^k - \sum_{j = 1}^{m}\lambda_j^k\left[\nabla^2F_j(x^k)d^k +\nabla F_j(x^k)\right]^\top d^k \\
		& = \sum_{j = 1}^{m}\lambda_j^k(d^k)^\top\nabla^2F_j(x^k)d^k - \left[\nabla F_{\lambda}^k(x^k) + \HF_{\lambda}^k(x^k)d^k\right]^\top d^k.
	\end{align*} 	
	Thus, by \eqref{lambdaN}, \eqref{sconvex}, and \eqref{DM3}, we obtain
	\[ -\sum_{j = 1}^{m}\lambda_j^k\nabla F_j(x^k)^\top d^k \geq \underline{L}\|d^k\|^2 + o(\|d^k\|^2) = \|d^k\|^2\left[\underline{L} +  \dfrac{o(\|d^k\|^2)}{\|d^k\|^2}\right]. \]
	Hence,  taking into account \eqref{limd} and \eqref{thetaalt}, for $k$ sufficiently large, it follows that
	\begin{equation}\label{dinamarca}
		-2\theta(x^k)=-\sum_{j = 1}^{m}\lambda_j^k\nabla F_j(x^k)^\top d^k \geq \dfrac{\underline{L}}{2}\|d^k\|^2.
	\end{equation}
	On the other hand, applying the Mean Value Theorem to the scalar function $\nabla F_{\lambda}^k(\cdot)^\top d^k$, there exists $v^k\coloneqq x^k+t_kd^k$ for some $t_k\in(0,1)$ such that
	\[ \nabla F_{\lambda}^k(x^k + d^k)^\top d^k =\nabla F_{\lambda}^k(x^k)^\top d^k + (d^k)^\top \nabla^2F_{\lambda}^k(v^k)d^k.\]
	Therefore, 
	\begin{equation*}\label{ucrania}
		\dfrac{|\nabla F_{\lambda}^k(x^k + d^k)^\top  d^k|}{\|d^k\|^2} \leq \dfrac{\|\nabla F_{\lambda}^k(x^k)+ \nabla^2F_{\lambda}^k(x^k)d^k\|}{\|d^k\|}+\|\nabla^2F_{\lambda}^k(v^k)-\nabla^2F_{\lambda}^k(x^k)\|. 
	\end{equation*}
	Now, by the definitions of $F_{\lambda}^k$ and $v^k$, and considering \eqref{lambdaN} and \eqref{limd}, we obtain
	\begin{align*}
		\lim_{k \to \infty} \|\nabla^2F_{\lambda}^k(v^k)-\nabla^2F_{\lambda}^k(x^k)\| 
		& \leq   \lim_{k \to \infty} \sum_{j=1}^m\lambda_j^k \|\nabla^2F_j(x^k+t_kd^k)-\nabla^2F_j(x^k)\| \\
		& \leq   \lim_{k \to \infty} \max_{j=1,\ldots,m} \|\nabla^2F_j(x^k+t_kd^k)-\nabla^2F_j(x^k)\|=0.
	\end{align*} 
	Thus, combining the latter two inequalities with \eqref{DM3}, we have
	\[\lim_{k \to \infty}  \dfrac{|\nabla F_{\lambda}^k(x^k + d^k)^\top  d^k|}{\|d^k\|^2}=0. \]
	Hence, for $k$ large enough, we have
	\[ |\nabla F_{\lambda}^k(x^k + d^k)^\top  d^k| \leq \sigma\dfrac{\underline{L}}{4}\|d^k\|^2, \]
	which, together with \eqref{dinamarca}, yields
	\[ \sum_{j=1}^m\lambda_j^k \nabla F_j(x^k + d^k)^\top  d^k=\nabla F_{\lambda}^k(x^k + d^k)^\top  d^k \geq - \sigma\dfrac{\underline{L}}{4}\|d^k\|^2 \geq \sigma\theta(x^k).\]
	Therefore, by the definition of $\D(\cdot,\cdot)$, \eqref{lambdaN}, and Lemma~\ref{l:pareto}{\it (ii)}, we obtain
	\[\D(x^k+d^k,d^k)\geq  \sum_{j=1}^m\lambda_j^k \nabla F_j(x^k + d^k)^\top  d^k \geq \sigma \theta(x^k) \geq \sigma \D(x^k,d^k), \]
	for all $k$ sufficiently large, concluding the proof.
\end{proof}

We require an additional auxiliary result.

\begin{lemma}\label{l:DS}
		Suppose that Assumption~\ref{assumption2} holds. Then,
$$\|\nabla F_{\lambda}^k(x^{k+1})-\nabla F_{\lambda}^k(x^k)-\nabla^2F_{\lambda}^k(\xs)(x^{k+1} - x^k)\|   \leq M \|x^{k+1} - x^k\| \varepsilon_k,$$
where $\varepsilon_k:=\max\{\|x^{k+1}-\xs\|^{\nu},\|x^{k}-\xs\|^{\nu}\}$.
\end{lemma}
\begin{proof}
By the definition of $F_{\lambda}^k$ in \eqref{alemanha} and taking into account \eqref{lambdaN}, we obtain
\begin{multline*}
\|\nabla F_{\lambda}^k(x^{k+1})-\nabla F_{\lambda}^k(x^k)-\nabla^2F_{\lambda}^k(\xs)(x^{k+1} - x^k)\| \\
 \leq \max_{j=1,\ldots,m} \|\nabla F_j(x^{k+1}) - \nabla F_j(x^k) - \nabla^2F_j(\xs)(x^{k+1}-x^k)\|.
 \end{multline*}
On the other hand, for each $j\in\{1,\ldots,m\}$, using \eqref{uruguai} and \eqref{lip}, we have
$$\|\nabla F_j(x^{k+1}) - \nabla F_j(x^k) - \nabla^2F_j(\xs)(x^{k+1}-x^k)\|   \leq  \int_{0}^{1}\| \left(\HF_j(x^k + \tau s^k)  - \nabla^2F_j(\xs)\right)s^k\|d\tau$$
$$\leq M \|s^k\| \int_{0}^{1} \| x^k + \tau s^k  - \xs\|^\nu d\tau \leq M \|s^k\|\max\{\|x^{k+1}-\xs\|^{\nu},\|x^{k}-\xs\|^{\nu}\}.$$
By combining the last two inequalities, we obtain the desired result.
\end{proof}

Now, we can establish the superlinear convergence of Algorithm~\ref{alg:MBFGS}.

\begin{theorem}\label{teo:superlinear}
Suppose that Assumptions~\ref{assumption2} and \ref{assumption3} hold. Then,  $\{x^k\}$ converges to $\xs$ superlinearly.
\end{theorem}
\begin{proof}
According to Theorem~\ref{teo:unit}, $d^k=x^{k+1}-x^k$ for all $k$ sufficiently large.  
Consequently, $B_{\lambda}^k (x^{k+1} - x^k) = - \nabla F_{\lambda}^k(x^k)$  (see \eqref{dN}), and hence
$$(B_{\lambda}^k-\nabla^2F_{\lambda}^k(\xs))(x^{k+1} - x^k)= \nabla F_{\lambda}^k(x^{k+1}) -\nabla F_{\lambda}^k(x^k)-\nabla^2F_{\lambda}^k(\xs)(x^{k+1} - x^k)-\nabla F_{\lambda}^k(x^{k+1}),$$
for all $k$ sufficiently large.
Therefore,
\begin{multline*}
\dfrac{\|\nabla F_{\lambda}^k(x^{k+1})\|}{\|x^{k+1}-x^k\|} \leq \dfrac{\|(B_{\lambda}^k-\nabla^2F_{\lambda}^k(\xs))(x^{k+1} - x^k)\|}{\|x^{k+1}-x^k\|} \\
+\dfrac{\|\nabla F_{\lambda}^k(x^{k+1})-\nabla F_{\lambda}^k(x^k)-\nabla^2F_{\lambda}^k(\xs)(x^{k+1} - x^k)\|}{\|x^{k+1}-x^k\|},
 \end{multline*}
	for all $k$ sufficiently large.
Taking limits on both sides of the latter inequality, using \eqref{DM2} and Lemma~\ref{l:DS}, we get
	\begin{equation}\label{grecia}
		\lim_{k \to \infty} \dfrac{\|\nabla F_{\lambda}^k(x^{k+1})\|}{\|x^{k+1}-x^k\|}=0.
	\end{equation}
	On the other hand, considering the definition of $F_{\lambda}^k$ in \eqref{alemanha}, Lemma~\ref{l:vtheta}{\it (iv)}, and Lemma~\ref{peru}{\it (i)}, we find that
	\begin{align*}
		\dfrac{\|\nabla F_{\lambda}^k(x^{k+1})\|}{\|x^{k+1}-x^k\|} & \geq \dfrac{\|\sum_{j = 1}^{m}\lambda_j^k \nabla F_{j}(x^{k+1})\|}{\|x^{k+1} -\xs\| + \|x^k-\xs\|} \geq \dfrac{\|\dsd(x^{k+1})\|}{\|x^{k+1} -\xs\| + \|x^k-\xs\|} \\
		& \geq \dfrac{\underline{L}}{2}\dfrac{\|x^{k+1} -\xs\| }{\|x^{k+1} -\xs\| + \|x^k-\xs\|}= \dfrac{\underline{L}}{2}\dfrac{1}{1+ \frac{\|x^k-\xs\|}{\|x^{k+1} -\xs\|}}.
	\end{align*} 
	Therefore, by using \eqref{grecia}, we conclude that
	\[
	\lim_{k \to \infty}  \dfrac{\|x^{k+1}-\xs\|}{\|x^k -\xs\|}=0,
	\]
which completes the proof.
\end{proof}

\begin{proofof}{Theorem~\ref{teo:convergence}}
The proof follows straightforwardly from Theorems~\ref{teo:unit} and \ref{teo:superlinear}.
\end{proofof}


\end{document}